\documentclass[11pt, a4paper, oneside,reqno]{amsart}

\usepackage[english]{babel}
\usepackage{amsmath, amsthm, amsfonts, mathrsfs, amssymb}
\usepackage{mathtools}
\mathtoolsset{centercolon}
\usepackage{booktabs}
\usepackage[shortlabels]{enumitem}
\usepackage[colorlinks, citecolor = blue]{hyperref}
\usepackage{todonotes}

\usepackage{fullpage}
\usepackage{comment}

\usepackage{color}
\usepackage{graphicx}
\usetikzlibrary{arrows}

\setlist[itemize]{leftmargin=30pt}
\setlist[enumerate]{leftmargin=30pt}
\newcommand{\Z}{\ensuremath{\mathbb{Z}}}

\newcommand{\R}{\ensuremath{\mathbb{R}}}
\newcommand{\C}{\ensuremath{\mathbb{C}}}

\newcommand{\mc}{\mathcal}

\DeclarePairedDelimiter\abs{\lvert}{\rvert}

\DeclarePairedDelimiter\cbrace\{\}
\DeclarePairedDelimiter\ha()
\DeclarePairedDelimiter{\ip}\langle\rangle
\DeclarePairedDelimiter{\nrm}\lVert\rVert

\newcommand{\nrmb}[1]{\bigl\|#1\bigr\|}
\newcommand{\absb}[1]{\bigl|#1\bigr|}
\newcommand{\hab}[1]{\bigl(#1\bigr)}
\newcommand{\cbraceb}[1]{\bigl\{#1\bigr\}}
\newcommand{\ipb}[1]{\bigl\langle#1\bigr\rangle}
\newcommand{\bracb}[1]{\bigl[#1\bigr]}

\newcommand{\nrms}[1]{\Bigl\|#1\Bigr\|}
\newcommand{\abss}[1]{\Bigl|#1\Bigr|}
\newcommand{\has}[1]{\Bigl(#1\Bigr)}
\newcommand{\cbraces}[1]{\Bigl\{#1\Bigr\}}

\DeclareMathOperator{\ind}{\mathbf{1}}

\newcommand{\dd}{\hspace{2pt}\mathrm{d}}
\newcommand{\dx}{\hspace{2pt}\mathrm{d}x}

\newcommand{\ApqO}{A_{p,q}^\alpha(Q_0)}
\newcommand{\Apq}{A_{p,q}^\alpha(Q)}
\newcommand{\avg}[3]{\ip{#1}_{#2,#3}}

\newtheorem{theorem}{Theorem}
\newtheorem{corollary}[theorem]{Corollary}
\newtheorem{lemma}[theorem]{Lemma}
\newtheorem{proposition}[theorem]{Proposition}

{}

\theoremstyle{remark}
\newtheorem{remark}[theorem]{Remark}

\theoremstyle{definition}

\numberwithin{theorem}{section}
\numberwithin{equation}{section}

\allowdisplaybreaks

\title{The two-weight fractional Poincar\'e--Sobolev sandwich}

\author[Lorist]{Emiel Lorist}
\address{Emiel Lorist \hfill\break\indent 
Delft Institute of Applied Mathematics \hfill\break\indent
Delft University of Technology \hfill\break\indent
P.O. Box 5031 \hfill\break\indent
2600 GA Delft, The Netherlands}
\email{e.lorist@tudelft.nl}

\author[C.C.M.L. Wagenaar]{Carel Wagenaar}
\address{Carel Wagenaar \hfill\break\indent 
Delft Institute of Applied Mathematics \hfill\break\indent
Delft University of Technology \hfill\break\indent
P.O. Box 5031 \hfill\break\indent
2600 GA Delft, The Netherlands}
\email{c.c.m.l.wagenaar@tudelft.nl}

\thanks{The authors would like to thank \'Oscar Dom\'inguez for pointing out \cite{DLTYY24} and Carlos P\'erez for several other reference suggestions. The first author was partially financed by the Dutch Research Council (NWO) on the project ``The sparse revolution for stochastic partial differential equations'' with project number \href{https://doi.org/10.61686/ZGRMR99948}{VI.Veni.242.057}. 
}


\begin{document}

 \begin{abstract}
We establish a two-weight fractional Poincar\'e--Sobolev sandwich, consisting of
a two-weight fractional Poincar\'e--Sobolev inequality and a two-weight embedding from the first-order Sobolev space to a Triebel--Lizorkin space defined via a difference norm. Our constants are asymptotically sharp as the fractional parameter approaches \(1\).
Our results are new even in the one-weight case.

For each inequality we give explicit quantitative dependence on Muckenhoupt weight characteristics and treat both subcritical and critical regimes, the former via elementary methods and the latter via sparse domination. As one of our main tools, we establish a new sparse domination result for Triebel--Lizorkin difference norms. Our methods unify, simplify and significantly extend various earlier approaches.
 \end{abstract}

\maketitle
\section{Introduction}

Poincar\'e and Poincar\'e--Sobolev inequalities are fundamental tools in the study of PDEs. For instance, in the classical De Giorgi--Nash--Moser scheme they play a key role in proving local
H\"older regularity for weak solutions of elliptic equations.
Weighted versions of these inequalities are central to the analysis of degenerate elliptic equations, beginning with the work of
Fabes, Kenig and Serapioni \cite{FKS82} (see \cite{HKM06} for an overview).
Their fractional analogues are obtained by replacing gradients with suitable difference seminorms. These inequalities are important both for finer local regularity questions and for nonlocal problems. Moreover, they are closely connected to the classical first-order theory through Bourgain--Brezis--Mironescu (BBM) type limits as the fractional order tends to $1$; see \cite{BBM01}.
The aim of this paper is to develop a
local two-weight theory that recovers and significantly extends these classical and fractional estimates in
a unified manner, with explicit quantitative dependence on the relevant weight
characteristics. We will formulate our results for cubes, from which extensions to, e.g., John domains follow by standard covering arguments \cite{DRS10}.

\medskip

For a cube $Q\subseteq \R^d$, an exponent $p \in [1,\infty)$, and a weight $\sigma\colon Q \to (0,\infty)$, we define the weighted Lebesgue space $L^p_\sigma(Q)$ as the space of all measurable functions $f \colon Q \to \C$ such that
\begin{align*}
	\nrm{f}_{L^p_\sigma(Q)} :=
	\has{\int_{Q} |f(x)|^p\sigma(x)^p\dd x}^{1/p}<\infty, 
\end{align*} 
where we stress that the weight enters as a multiplier, i.e. through $\sigma^p$ in the norm. For $r\in [1,\infty)$ and $s \in (0,1)$, we furthermore define the weighted Triebel--Lizorkin space $F^{s,\sigma}_{p,r}(Q)$ as the space of all $f \in L^p_\sigma(Q)$ such that
	\begin{align*}
	[f]_{F^{s,\sigma}_{p,r}(Q)}:= \has{\int_{Q} \left(\int_{Q} \frac{|f(x)-f(y)|^r}{|x-y|^{d+sr}}\mathrm{d}y\right)^{p/r}\sigma(x)^p\dd x}^{1/p}<\infty, 
	\end{align*}
    see, e.g., \cite{BSY23, Pra19, Tr83} for the connection to Triebel--Lizorkin spaces defined via Littlewood--Paley theory. 
    
Defining $\ip{f}_Q:=\frac{1}{|Q|}\int_Q f$, our main goal is to study the following sandwich of the $F^{s,\sigma}_{p,r}(Q)$-seminorm between quantities of zero- and first-order smoothness:
    \begin{align} 
		\nrm{f-\ip{f}_{Q}}_{L^{p_0}_{\sigma_0}(Q)} &\lesssim (1-s)^{\frac1r}  \cdot [f]_{F^{s,\sigma_s}_{p_s,r}(Q)}\lesssim   \nrmb{\abs{\nabla f}}_{L^{p_1}_{\sigma_1}(Q)}\label{eq:sandwich}
	\end{align}
under suitable conditions on $p_j,r,\sigma_j$ for $j\in \cbrace{0,s,1}$. The factor $(1-s)^{1/r}$ is the sharp BBM-factor mentioned above. Although combining the two parts of the sandwich in \eqref{eq:sandwich} yields a nonfractional two-weight Poincar\'e--Sobolev inequality, a direct treatment of this estimate is both quantitatively sharper and conceptually simpler. We will therefore  study the following three inequalities separately:
\begin{align}
\nrm{f-\ip{f}_{Q}}_{L^q_\omega(Q)} &\lesssim \nrmb{\abs{\nabla f}}_{L^p_\sigma(Q)}, \label{eq:introAC}\\
\nrm{f-\ip{f}_{Q}}_{L^q_\omega(Q)} &\lesssim (1-s)^{\frac1r}[f]_{F^{s,\sigma}_{p,r}(Q)}, \label{eq:introAB}\\
(1-s)^{\frac1r}[f]_{F^{s,\omega}_{q,r}(Q)} &\lesssim \nrmb{\abs{\nabla f}}_{L^p_\sigma(Q)}, \label{eq:introBC}
\end{align}
under suitable conditions on $p,q,r,\sigma,\omega$. 

The natural hypothesis in these inequalities is a local two-weight Muckenhoupt condition. For $\alpha\geq 0$, we write $(\omega,\sigma)\in A_{p,q}^\alpha(Q)$ if
\begin{align*}
[\omega,\sigma]_{A_{p,q}^\alpha(Q)}
:=\sup_{R\subseteq Q} |R|^\alpha \ip{\omega}_{q,R}\ip{\sigma^{-1}}_{p',R}<\infty,
\end{align*}
where the supremum is taken over all cubes $R \subseteq Q$ and $$\ip{f}_{p,Q}:= \begin{cases} \ip{|f|^p}_Q^{1/p} \quad &p <\infty,\\
 \nrm{f}_{L^\infty(Q)} \quad &p=\infty.\end{cases}$$  In the special case $\alpha=0$ and $p=q$, this is a localized version of the classical Muckenhoupt $A_p$-condition. 
 
Define
\[
\varepsilon:= \tfrac{t}{d}-\bigl(\tfrac1p-\tfrac1q\bigr)-\alpha,
\]
where $t$  denotes the difference in smoothness between the two sides of the inequality under consideration, i.e. $t=1$ for \eqref{eq:introAC}, $t=s$ for \eqref{eq:introAB} and $t=1-s$ for \eqref{eq:introBC}.
Then $\varepsilon$ quantifies the scaling deficit between the two sides of the inequality. We will call the case $\varepsilon=0$ the \emph{critical case}, which is where we obtain our most delicate estimates. We call  $\varepsilon>0$ the \emph{subcritical case}. Although this case may  be deduced from the critical one, it  allows for substantially simpler arguments, since the positive scaling deficit $\varepsilon>0$ makes it possible to sum across dyadic scales. We therefore treat the critical and subcritical regimes separately, leading to slightly larger admissible parameter ranges in the subcritical case.

Our main results may be summarized as follows. For the precise statements, we refer to Sections \ref{sec:AC}-\ref{sec:BC}.
\begin{enumerate}[(i)]
    \item In Theorem \ref{theorem:IleqIII} we prove the two-weight
    Poincar\'e--Sobolev inequality \eqref{eq:introAC}. Our claims to novelty in this theorem are rather mild. Instead, we would like to emphasize that our proofs unify and significantly simplify the existing literature.
    \item Theorem \ref{theorem:IleqII} proves the two-weight fractional
    Poincar\'e--Sobolev inequality \eqref{eq:introAB} with the sharp
    BBM-factor. As far as we are aware,
    our results with either $p\neq r$ or $\omega^q\neq \sigma^p$ are completely new. Our results for $p=r$ and $\omega^q= \sigma^p$ recover and extend various previous results.
    \item Theorem \ref{theorem:IIleqIII} establishes the two-weight Sobolev to Triebel--Lizorkin embedding
    \eqref{eq:introBC} with the sharp BBM-factor, in which the cases where either  $p\neq q$ or  $\omega^q \neq \sigma^p$ are new and the cases where $p=q$ and $\omega^q= \sigma^p$ again recover and extend various previous results.
\end{enumerate}
In all three inequalities, we will explicitly track the dependence on the $A_{p,q}^\alpha(Q)$-characteristic and, whenever needed, additional $A_\infty$-characteristics of $\omega$ and $\sigma$. Furthermore, we  derive one-weight consequences in Corollaries \ref{cor:AleqConeweight}, \ref{cor:AleqBoneweight} and \ref{cor:BleqConeweight}, which also contain many new cases. For a thorough comparison to the literature, we refer to Subsections \ref{sec:ACliterature}, \ref{sec:ABliterature} and \ref{subs:BCliterature}. Here we would like to note that weighted, fractional Poincar\'e--Sobolev inequalities have attracted a lot of attention in recent years and our results recover and significantly generalize results from several recent works, including \cite{HLYY25, HMPV23b, HMPV25, KV21, MPW24, PR19}.

\medskip

To illustrate our general proof strategy, let us briefly revisit the classical $(p,p)$-Poincar\'e inequality. By a dyadic telescoping argument combined with the Lebesgue differentiation theorem (see \eqref{eq:basicdom} for more details) and the $(1,1)$-Poincar\'e inequality, one has for a.e. $x\in Q$ that 
\begin{align*}
|f(x)-\ip{f}_Q|
&\overset{\text{\ref{it:a}}}{\lesssim} \sum_{R\in \mc{D}(Q)} \ip{|f-\ip{f}_R|}_R \ind_R(x)
\overset{\text{\ref{it:b}}}{\lesssim} \sum_{R\in \mc{D}(Q)} \ell(R)\ip{|\nabla f|}_R \ind_R(x),
\end{align*}
where $\mc{D}(Q)$ denotes the collection of all dyadic subcubes of $Q$ and $\ell(R)$ denotes the side length of $R$. By taking $L^p$-norms and using that dyadic cubes of a fixed side length are pairwise disjoint, we obtain
\begin{align}\tag*{\ref{it:c}}
\begin{aligned}
\nrm{f-\ip{f}_Q}_{L^p(Q)}
&\lesssim_d 
\nrms{\sum_{j=0}^\infty
\sum_{\substack{R\in \mc{D}(Q):\\ \ell(R)=2^{-j}\ell(Q)}}
\ell(R)\ip{|\nabla f|}_R \ind_R
}_{L^p(Q)} \\
&\leq \ell(Q)\sum_{j=0}^\infty 2^{-j}
\has{
\sum_{\substack{R\in \mc{D}(Q):\\ \ell(R)=2^{-j}\ell(Q)}}
\nrm{\nabla f}_{L^p(R)}^p
}^{1/p}
=  \ell(Q)\nrm{\abs{\nabla f}}_{L^p(Q)}.
\end{aligned}
\end{align}
This argument serves as a prototype for all of our results. In each case, the proof is organized around three ingredients:
\begin{enumerate}[(a)]
    \item\label{it:a} a domination principle;
    \item\label{it:b} an unweighted version of the inequality under consideration;
    \item\label{it:c} a norm estimate for the resulting dyadic object.
\end{enumerate}
In the prototypical argument above, we have
\[
\varepsilon=\tfrac{1}{d}-\bigl(\tfrac1p-\tfrac1p\bigr)-0=\tfrac1d>0.
\]
This  makes step \ref{it:c} straightforward, since it yields summable decay across dyadic scales. As a consequence, the domination step \ref{it:a} can also be taken in a very simple form. This is typical for our results in the subcritical regime.

In contrast, in the critical case $\varepsilon=0$, there is no decay across dyadic scales available in step \ref{it:c}. One therefore needs a much more refined version of the domination step \ref{it:a}. Our arguments will rely on the modern harmonic-analytic technique of sparse domination, which first appeared in the context of Poincar\'e--Sobolev inequalities in \cite{KV21,LLO21}. For the (fractional) Poincar\'e--Sobolev inequalities \eqref{eq:introAC} and \eqref{eq:introAB}, the required sparse domination principle is by now standard; see Lemma~\ref{lem:sparsedom}. For \eqref{eq:introBC}, however, we prove a novel sparse domination principle  for the difference quotients
$$
\has{\int_Q \frac{|f(x)-f(y)|^r}{|x-y|^{d+sr}}\dd y}^{1/r}, \qquad x \in Q
$$
in Theorem~\ref{thm:fracsparsedom}, which is one of our main contributions. Once such a domination principle is available, step \ref{it:b} reduces to the corresponding unweighted estimate, while step \ref{it:c} follows from weighted norm inequalities for sparse operators, see Proposition~\ref{prop:weaksparse}.

\medskip

This paper is organized as follows. In Section \ref{sec:weights&sparse}, we  list properties of the Muckenhoupt $\Apq$-class and relate this class to the classical $A_p$ Muckenhoupt classes. Moreover, we introduce (fractional) sparse operators and show boundedness of these operators and the (fractional) maximal operator. In Section \ref{sec:dom}, we discuss known and prove new domination principles required for step \ref{it:a} in the proof of our main results.  Sections \ref{sec:AC}, \ref{sec:AB} and \ref{sec:BC} will be about inequalities \eqref{eq:introAC}, \eqref{eq:introAB} and \eqref{eq:introBC} respectively and have a similar structure. We start each of those sections with our main two-weight result and extract a one-weight corollary afterwards. We end each of these  sections with a comparison to the existing literature. Finally, in Section \ref{sec:extensions}  we comment on several directions in which the two-weight fractional Poincar\'e--Sobolev sandwich can be extended and end with an appendix on the truncation method for (fractional) Poincar\'e--Sobolev inequalities.

	\section{Weights and sparse operators} \label{sec:weights&sparse}
In this section, we will start by defining the Muckenhoupt weight classes we will use and discuss their properties. Afterwards, we turn to sparse operators and their weighted norm estimates. 

\subsection{Weighted function spaces}
Let $p,r \in [1,\infty)$ and $s \in (0,1)$. Besides the weighted Lebesgue space $L^p_\sigma(Q)$ and weighted Triebel-Lizorkin space $F^{s,\sigma}_{p,r}(Q)$ defined in the introduction, we define the
weighted weak Lebesgue space $L^{p,\infty}_\sigma(Q)$ as the space of measurable $f \colon Q \to \R$ such that
\begin{align*}
  \nrm{f}_{L^{p,\infty}_\sigma(Q)}&:= \sup_{\lambda>0} \,\nrm{\lambda\cdot \ind_{\cbrace{\abs{f}>\lambda}}}_{L^p_\sigma(Q)}\\&\phantom{:}=\sup_{\lambda>0}\, \lambda \cdot \sigma^p\hab{\cbrace{x \in Q:\abs{f(x)}>\lambda}}^{\frac{1}{p}}<\infty,
\end{align*}
where the measure $\sigma^p$ is defined by $\sigma^p(A) := \int_A \sigma^p\dx$ for measurable $A \subseteq \R^d$. Note that $\nrm{f}_{L^{p,\infty}_\sigma(Q)}\neq \nrm{f\sigma}_{L^{p,\infty}(Q)}$. 

For a weight $\sigma$ such that $\sigma^{-1}\in L^{p'}(Q)$, we define the weighted Sobolev space $W^{1,p}_\sigma(Q)$ as the space of all  $f \in L^p_\sigma(Q)\subseteq L^1(Q)$ such that the distributional derivative $\partial_j f$ lies in $L^p_\sigma(Q)$ for $1\leq j\leq d$ with
$$
\nrm{f}_{W^{1,p}_\sigma(Q)} := \nrm{f}_{L^p_\sigma(Q)}+ \nrmb{\abs{\nabla f}}_{L^p_\sigma(Q)}.
$$
We stress once more that our normalization of the weight in $L^p$-spaces is a \emph{multiplier}, in contrast to the \emph{change of measure} given by  
$$
\nrm{f}_{L^p(Q,w)}:= \has{\int_{Q}\abs{f}^p w\dd x}^{\frac{1}{p}}.
$$
We will emphasize this difference by using weights $(\omega,\sigma)$ whenever we use the multiplier normalization, whereas we will use weights $w$ whenever we use the change of measure normalization.

\subsection{Muckenhoupt weights}  Fix a cube $Q \subseteq \R^d$, let $p,q \in [1,\infty)$ and $\alpha \geq 0$. 
We will discuss a few basic properties of the Muckenhoupt $A^\alpha_{p,q}(Q)$-class, which we defined in the introduction as the class of weights $(\omega,\sigma)$ such that 
\begin{align*}
[\omega,\sigma]_{A_{p,q}^\alpha(Q)}
:=\sup_{R\subseteq Q} |R|^\alpha \ip{\omega}_{q,R}\ip{\sigma^{-1}}_{p',R}<\infty.\end{align*}
Note that $\Apq$ would be empty for $\alpha < 0$ by the Lebesgue differentiation theorem. Furthermore, for all $\alpha \geq \frac1q + \frac1{p'}$ and $R\subseteq Q$ we have
\begin{align*}
        \abs{R}^\alpha  \ip{\omega}_{q,R}\ip{\sigma^{-1}}_{p',R}   &= \abs{R}^{\alpha-\frac1q-\frac1{p'}} \omega^q(R)\sigma^{-p'}(R) \leq \abs{Q}^{\alpha-\frac1q-\frac1{p'}} \omega^q(Q)\sigma^{-p'}(Q) \\
        &\leq \abs{Q}^{\alpha-\beta} [\omega,\sigma]_{A^\beta_{p,q}(Q)},
    \end{align*}
    and thus for $\beta = \frac1q + \frac1{p'}$
\begin{equation*}
     [\omega,\sigma]_{A^\alpha_{p,q}(Q)} = \abs{Q}^{\alpha-\beta} [\omega,\sigma]_{A^\beta_{p,q}(Q)}.   
\end{equation*}
Therefore, it is only useful to consider $0 \leq \alpha \leq \frac1q + \frac1{p'}$. Below we note some further basic properties of $A^\alpha_{p,q}(Q)$.

 \begin{lemma}Let $Q\subseteq \R^d$ be a cube, let $1\leq p\leq q <\infty$, $0 \leq \alpha \leq \frac1q + \frac1{p'}$ and $(\omega,\sigma)\in \Apq$. \label{lem:propertiesMuckenhoupt}
 \begin{enumerate}[(i)]
     \item \label{it:weightalphabeta} For  $\beta<\alpha$ we have $A_{p,q}^{\beta}(Q)\subsetneq A_{p,q}^{\alpha}(Q)$ with
\begin{equation*}
    [\omega,\sigma]_{A_{p,q}^{\alpha}(Q)} \leq |Q|^{\alpha-\beta} \cdot[\omega,\sigma]_{A_{p,q}^{\beta}(Q)}.
\end{equation*}			
\item \label{it:weightpq} For $p \leq p_0 \leq q_0 \leq q$  we have $ A_{p,q}^{\alpha}(Q)\subseteq A_{p_0,q_0}^{\alpha}(Q)$ with
\begin{equation*}
    [\omega,\sigma]_{A_{p_0,q_0}^{\alpha}(Q)} \leq  [\omega,\sigma]_{A_{p,q}^{\alpha}(Q)}.
\end{equation*}		
\item \label{it:alphat} For $1 \leq t \leq \min\cbrace{p,q}$ we have
\begin{equation*}
     [\omega^{t},\sigma^{t}]_{A_{p/t,q/t}^{\alpha t}(Q)}^{\frac1t}=[\omega,\sigma]_{A_{p_0,q}^\alpha(Q)},
\end{equation*}
where $\frac{1}{p_0} = \frac{1}{p}+\frac{1}{t'}$.
 \end{enumerate}
 \end{lemma}

 \begin{proof}
    For \ref{it:weightalphabeta}, the inequality follows from
    \[
        \abs{R}^\alpha= \abs{R}^{\alpha-\beta}\abs{R}^\beta \leq \abs{Q}^{\alpha-\beta}  \abs{R}^\beta.
    \]
    For the strict inclusion take $$\ha{\omega,\sigma}=\hab{|x|^{-\gamma_1d},|x|^{\gamma_2d}},$$ where $\gamma_1\in (0,\frac1q)$ and $\gamma_2 \in [0,\frac{1}{p'})$ such that $\beta < \gamma_1+\gamma_2 < \alpha$. Then $(\omega,\sigma)\in A^\alpha_{p,q}(Q)\backslash A^\beta_{p,q}(Q)$.
    
    The second statement \ref{it:weightpq} follows from H\"older's inequality. For \ref{it:alphat}, the equality easily follows after noting that
    \[
        \frac{1}{t(p/t)'}= \frac1t - \frac1p= \frac1{p_0'}. \qedhere 
    \]
 \end{proof}
For a cube $Q\subseteq \R^d$ we say that $w$ belongs to the Muckenhoupt $A_\infty(Q)$-class and write $w \in A_{\infty}(Q)$ if
$$
[w]_{A_{\infty}(Q)}:=\sup_{R\subseteq Q}\frac{1}{w(R)}\int_R \sup_{S\subseteq R} \ip{w}_{1,S} \ind_S(x) \dd x,
$$
where we omit $Q$ if $[w]_{A_{\infty}(Q)}<\infty$ for all $Q\subseteq \R^d$. By H\"older's inequality, we know that for $w\in A_\infty(Q)$ we have $w^t \in A_\infty(Q)$ for all $t \in (0,1]$ with
\begin{equation}\label{eq:Ainftyembeds}
    [w^t]_{A_\infty(Q)} \leq [w]_{A_\infty(Q)}^t.
\end{equation}
As discussed above, the relevant range for $\alpha$ in $A^{\alpha}_{p,q}(Q)$ is $[0,\tfrac{1}{q}+\tfrac1{p'}]$. If we also have that either $\omega^q\in A_\infty(Q)$ or $\sigma^{-p'}\in A_\infty(Q)$, then the relevant range reduces to $[0,\tfrac{1}{q}+\tfrac1{p'})$, as shown in the following proposition.
\begin{proposition}\label{prop:AbetaAinfty}
Let $Q\subseteq \R^d$ be a cube, let $1\leq p\leq q <\infty$ and $\alpha = \frac1q+\frac1{p'}$. Suppose that $(\omega,\sigma) \in A^\alpha_{p,q}$ with 
either $\omega^q\in A_\infty(Q)$ or $p>1$ and $\sigma^{-p'}\in A_\infty(Q)$. Then there exists a $0\leq \beta <\alpha$ such that $(\omega,\sigma) \in  A^{\beta}_{p,q}(Q)$ with
    \[
        [\omega,\sigma]_{A_{p,q}^\beta(Q)}\leq 2 \,|Q|^{\beta-\alpha} \cdot [\omega,\sigma]_{A_{p,q}^\alpha(Q)}. 
    \]
\end{proposition}
\begin{proof}
    Let $(\omega,\sigma)\in A^\alpha_{p,q}(Q)$ and assume $\omega^q\in A_\infty$, the proof for $\sigma^{-p'}\in A_\infty(Q)$ is similar. By \cite[Theorem 2.3]{HPR12}, there exists an $s>1$ such that 
    $$
\ip{\omega}_{qs,Q} \leq 2 \,\ip{\omega}_{q,Q}.
    $$
    Now take $\beta =\frac1{qs}+\frac{1}{p'}<\alpha$. Then, by H\"older's inequality, we have
    \begin{align*}
        |R|^\beta \ip{\omega}_{q,R}\ip{\sigma^{-1}}_{p',R} \leq |R|^\beta \ip{\omega}_{sq,R}\ip{\sigma^{-1}}_{p',R} &= \omega^{qs}(R)^{\frac1{qs}} \sigma^{-p'}(R)^{\frac{1}{p'}}\\
&\leq  \omega^{qs}(Q)^{\frac1{qs}} \sigma^{-p'}(Q)^{\frac{1}{p'}}\\   
&\leq |Q|^{\frac1{qs}+\frac1{p'}} \cdot  \ip{\omega}_{qs,Q} \ip{\sigma^{-1}}_{p',Q}\\  
&\leq 2 \,[\omega,\sigma]_{A_{p,q}^\alpha(Q)} |Q|^{\beta-\alpha}.
    \end{align*}
    Taking the supremum over all cubes $R\subseteq Q$ finishes the proof.
\end{proof} 
In view of Proposition \ref{prop:AbetaAinfty}, whenever $\omega^q\in A_\infty(Q)$ or $\sigma^{-p'}\in A_\infty(Q)$, we can take $\alpha <\frac1q+\frac1{p'}$ without excluding any admissible weights. Recall that in the two-weight fractional Poincar\'e--Sobolev inequalities under study, we will use the parameter
\[
    \varepsilon = \tfrac{t}{d} - (\tfrac{1}{p}-\tfrac1q)-\alpha,
\]
where $t$ denotes the difference in smoothness ($t=1$, $t=s$ or $t=1-s$, respectively). In the subcritical case $\varepsilon>0$ we have 
\[
    \alpha < \tfrac{t}{d} - (\tfrac{1}{p}-\tfrac1q) \leq 1- (\tfrac{1}{p}-\tfrac1q) = \tfrac1q+\tfrac{1}{p'}.
\]
In the critical case $\varepsilon=0$ we similarly have  $\alpha \leq \frac1q+\frac{1}{p'}$. However, in the critical cases we always have $\omega^q\in A_\infty(Q)$ or $\sigma^{-p'}\in A_\infty(Q)$, so that the endpoint case $\alpha= \frac1q+\frac{1}{p'}$ does not yield any additional weights. Therefore, we will assume in the remainder of this paper that $\alpha < \frac1q+\frac{1}{p'}$.

\medskip

The weight class $\Apq$ in the case $p=q \in [1,\infty)$ and $\alpha=0$ is a rescaled version of the classical Muckenhoupt $A_p$-class. Indeed, for a weight $w$ we say that  $w\in A_p$ if
	\[
	[w]_{A_p}=\sup_{Q} \avg{w}{1}{Q}\avg{w^{-1}}{\frac{1}{p-1}}{Q} < \infty,
	\]
where the supremum is taken over all cubes $Q \subseteq \R^d$.
	Note that for a fixed cube $Q\subseteq \R^d$ we have 
	\begin{align*}
		[w^{\frac1p},w^{\frac1p}]_{A^0_{p,p}(Q)} = \sup_{R\subseteq Q} \avg{w}{1}{R}^{\frac1p}\avg{w^{-1}}{\frac{1}{p-1}}{R}^{\frac1p} \leq [w]_{A_p}^{\frac1p},
	\end{align*}
	so that $w\in A_p$ implies $(w^\frac1p,w^{\frac1p})\in A^0_{p,p}(Q)$. It is well-known that $w \in A_p$ implies that $w \in A_\infty$. Therefore,  for $(w,w)\in A_{p,p}^0(Q)$ we have $w^p \in A_\infty(Q)$ as well. More generally, we have the following lemma.
	\begin{lemma}\label{lem:ApApq}
		Let $1\leq u\leq p\leq q<\infty$, $w \in A_u$ and let $\alpha  = (u-1) (\frac1p-\frac1q)$. Then we have for any cube $Q \subseteq \R^d$
		$$
		[w^\frac1q,w^{\frac1p}]_{A_{p,q}^\alpha(Q)} \leq [w]_{A_p}^{\frac{1}{p}} [w]_{A_u}^{\frac{1}{p}-\frac1q} \cdot \has{\frac{\abs{Q}^u}{w(Q)}}^{\frac1p-\frac1q}
		$$
		and thus $(w^\frac1q,w^{\frac1p}) \in A_{p,q}^\alpha(Q)$.
	\end{lemma}
	\begin{proof}
    By \cite[(7.2.1)]{Gra14} we know that for any cube $R \subseteq Q$  we have
	\begin{equation*}
		\has{\frac{\abs{R}}{\abs{Q}}}^u \leq [w]_{A_u}{\frac{w(R)}{w(Q)}}.
	\end{equation*}
    Therefore
		\begin{align*}
			\abs{R}^\alpha \ip{w^{\frac1q}}_{q,R}\ip{w^{-\frac1p}}_{p',R}&=\abs{R}^\alpha \ip{w}_{1,R}^{\frac1q}\ip{w^{-1}}_{\frac{1}{p-1},R}^\frac{1}{p}\\&\leq [w]_{A_p}^{\frac{1}{p}} \cdot \frac{\abs{R}^{\frac{u}p-\frac{u}q}}{w(R)^{\frac1p-\frac1q}}\leq [w]_{A_p}^{\frac{1}{p}} [w]_{A_u}^{\frac{1}{p}-\frac1q} \cdot \has{\frac{\abs{Q}^u}{w(Q)}}^{\frac1p-\frac1q}.
		\end{align*}
		Taking the supremum over all cubes $R\subseteq Q$ finishes the proof.
	\end{proof}

\subsection{Sparse operators}
For the critical case of our inequalities, we will need boundedness of \textit{sparse operators}. A collection $\mc{S}$ of cubes is called sparse if for all $Q \in \mc{S}$ there exists an $E_Q\subseteq Q$ such that $|E_Q|\geq \tfrac{1}{2}|Q|$ and the $E_Q$'s are pairwise disjoint.

	For a sparse collection of cubes $\mc{S}$, $r \in (0,\infty)$ and $\beta\in [0,1)$, we define the (fractional) sparse operator
	\[
	\mc{A}_{\mc{S}}^{r,\beta} f (x) := \has{\sum_{Q\in\mc{S}} \hab{\abs{Q}^{\beta} \ip{\abs{f}}_{Q}}^r\ind_Q(x)}^{\frac1r}, \qquad x \in \R^d.	
	\]
In the following proposition we collect the weighted estimates that we will need on $\mc{A}_{\mc{S}}^{r,\beta}$.

	\begin{proposition}\label{prop:weaksparse}
		Let $1\leq p\leq q<\infty$, $r \in (0,\infty)$, $0\leq \alpha < \frac1{q}+\frac1{p'}$ and define  $\beta := \alpha+ \tfrac{1}{p}-\tfrac1q< 1.$ 
 Let $Q \subseteq \R^d$ be a cube and let $\mc{S}\subseteq \mc{D}(Q)$ be a sparse collection of cubes.
		\begin{enumerate}[(i)]
			\item\label{it:strong} \emph{(Strong-type estimate)} Take $(\omega,\sigma) \in A_{p,q}^\alpha(Q)$ with  $\sigma^{-p'} \in A_{\infty}(Q)$. Then we have for $1<p\leq r$
$$			{\nrm{\mc{A}_{\mc{S}}^{r,\beta}}_{L^p_\sigma(Q) \to L^{q}_\omega(Q)}} \lesssim   {[\omega,\sigma]_{A_{p,q}^\alpha(Q)} } \cdot [\sigma^{-p'}]_{A_\infty(Q)}^{\frac{1}{q}},
$$
and if, in addition, $\omega^q \in A_{\infty}(Q)$, we have for $p>\max\cbrace{1,r}$
$$		{\nrm{\mc{A}_{\mc{S}}^{r,\beta}}_{L^p_\sigma(Q) \to L^{q}_\omega(Q)}}\lesssim  {[\omega,\sigma]_{A_{p,q}^\alpha(Q)} } \cdot 
\begin{cases}
				[\omega^q]_{A_\infty(Q)}^{\frac1r-\frac{1}{p}} + [\sigma^{-p'}]_{A_\infty(Q)}^{\frac{1}{q}}\quad & p<q \text{ or } \beta =0,\vspace{3pt}\\
				[\omega^q]_{A_\infty(Q)}^{\frac1r-\frac{1}{p}} \hspace{3pt}\cdot\hspace{3pt} [\sigma^{-p'}]_{A_\infty(Q)}^{\frac{1}{q}} & p=q \text{ and } \beta > 0.
			\end{cases}
            $$
			\item\label{it:weak} \emph{(Weak-type estimate)} Take $(\omega,\sigma) \in A_{p,q}^\alpha(Q)$ with  $\omega^q \in A_{\infty}(Q)$. Then we have
			$$
			{\nrm{\mc{A}_{\mc{S}}^{r,\beta}}_{L^p_\sigma(Q) \to L^{q,\infty}_\omega(Q)}}\lesssim {[\omega,\sigma]_{A_{p,q}^\alpha(Q)}} \cdot \begin{cases}
				[\omega^q]_{A_\infty(Q)}^{\frac{1}{p'}} \quad & 1=r<p \text{ and } (p<q \text{ or }  \beta =0),\\ \vspace{3pt}[\omega^q]_{A_\infty(Q)}^{\frac1r} \quad & \text{otherwise}.
			\end{cases}
			$$
		\end{enumerate}
        The implicit constants depend on $d$ and on $p,q,r,\alpha$, but are uniform whenever these parameters are bounded away from their endpoints.
	\end{proposition}

	\begin{proof}[Proof of Proposition \ref{prop:weaksparse}]
The strong-type estimate in \ref{it:strong} for the cases 
\begin{itemize}
    \item $p<q$ 
    \item $p=q\leq r$
    \item $\beta =0$
\end{itemize}
follows from \cite[Theorem 1.1]{FH18} using $\alpha = 1-\beta$, noting that $\omega^q \in A_\infty(Q)$ is not used in the proof when $p \leq r$.
For the case $p=q>r$ and $\beta>0$, we note that by \cite[Theorem 4.2]{NSS24} we have for $f \in L^p_\sigma(Q)$ that
		$$
		\nrmb{\mc{A}_{\mc{S}}^{r,\beta}f}_{L^{p}_\omega(Q)} \lesssim_{p,r} [\omega^q]_{A_\infty(Q)}^{\frac1r-\frac{1}{p}} \nrmb{\mc{A}_{\mc{S}}^{p,\beta}f}_{L^{p}_\omega(Q)}.
		$$
Therefore, this case also follows from 
\cite[Theorem 1.1]{FH18} using $r=p$ and $\alpha = 1-\beta$.

For the weak-type estimate in \ref{it:weak}, we start with a proof that works for any $p,q,r,\beta$. Take $f \in L^p_\sigma(Q)$. By \cite[Theorem 4.2]{NSS24} we have
		$$
		\nrm{\mc{A}_{\mc{S}}^{r,\beta}f}_{L^{q,\infty}_\omega(Q)} \lesssim [\omega^q]_{A_\infty(Q)}^{\frac1r} \nrmb{\,\sup_{R \in \mc{S}}\, {{\abs{R}^\beta \ip{f}_{1,R}}} \ind_R}_{L^{q,\infty}_\omega(Q)}.
		$$
        We can easily estimate the right-hand side, without using the sparsity of $\mc{S}$.
		Fix $\lambda>0$ and let $\cbrace{R_j}_j$ be the collection of maximal cubes $R \in \mc{S}$ such that $\abs{R}^\beta \ip{f}_{1,R}>\lambda$. Then we have by H\"older's inequality
		\begin{align*}
			\lambda^q \cdot \omega^q\has{\cbraces{\,\sup_{R \in \mc{S}} \abs{R}^\beta \ip{f}_{1,R} \ind_R >\lambda}} &= \lambda^q \sum_j \omega^q(R_j)\\&\leq \sum_j \frac{1}{\abs{R_j}^{(1-\beta) q}} \has{\int_{R_j} \abs{f}}^q \cdot  \omega^q(R_j)\\
			&\leq \sum_j \has{ \frac{\omega^q(R_j)^{\frac{1}{q}} \sigma^{-p'}(R_j)^{\frac{1}{p'}}}{\abs{R_j}^{1-\beta}} }^q \has{\int_{R_j}\abs{f}^p\sigma^p}^{\frac{q}{p}} \\
			&\leq  [\omega,\sigma]_{A_{p,q}^\alpha(Q)}^q \nrm{f}_{L^p_\sigma(Q)}^q,
		\end{align*}
        with the usual modifications if $p'=\infty$.
		Since $\lambda>0$ was arbitrary, this finishes the proof.
        
If $1=r<p$ and  either $p<q$ or $\beta = 0$, we can give a sharper estimate in terms of $[\omega^q]_{A_\infty(Q)}$. Indeed, by \cite[Theorem 1.8]{LSU09} it suffices to estimate
$$
\mc{T}^* := \sup_{R \in \mc{S}}\omega^q(R)^{-1/q'}
\nrms{\sum_{S \in \mc{S}:S \subseteq R}
\frac{\omega^q(S)}{\abs{S}^{1-\beta}} \ind_S}_{L^{p'}_{\sigma^{-1}}(Q)},
$$
which is done in \cite[Theorem 4.3]{FH18} using $\alpha = 1-\beta$. 
	\end{proof}
\begin{remark}
There are various other cases in Proposition \ref{prop:weaksparse}\ref{it:weak} for which sharper estimates in terms of $[\omega^q]_{A_\infty(Q)}$ are available. For example, when $1<p=q<r$ and $\beta = 0$, the result holds even without the assumption $\omega^q \in A_\infty(Q)$,  see \cite[Theorem 1.2]{HL18}. A logarithmic correction is known in case $p=q=1$ and $\beta=0$, see \cite[Theorem C]{NS24}. However, these cases play no role in the rest of this article, and therefore we chose to omit them from Proposition \ref{prop:weaksparse}\ref{it:weak}. 

To improve our main results, we would need a sharper estimate for the case $p=q$ and $\beta > 0$ in Proposition \ref{prop:weaksparse}\ref{it:weak}, which seems unavailable. Indeed, for example the proof of \cite[Theorem 1.2]{HL18} for $p<r$ does not extend to the fractional case  $\beta > 0$. 
We note that an improvement in this case was claimed in  \cite[Theorem 1]{HY20}. However, in the proof the collection $\mc{S}$ is sparsified further such that for $R\in \mc{S}$
$$
\sum_{S\in \mc{S}:S\subseteq R} \abs{S}^{1-\frac{\alpha}{d}} \leq \tfrac14 \,\abs{R}^{1-\frac{\alpha}{d}},
$$
which is possible if and only if $\alpha =0$.
\end{remark}

We will also use the following well-known corollary of Proposition \ref{prop:weaksparse}\ref{it:strong} for the local fractional maximal operator
\[
	M_{Q}^{\beta} f (x) := \sup_{R\in\mc{D}(Q)} {\abs{R}^{\beta} \ip{\abs{f}}_{R}}\ind_R(x),  \qquad x \in \R^d.	
	\]

\begin{corollary}\label{cor:strongmax}
		Let $1< p\leq q<\infty$, $0\leq \alpha < \frac1{q}+\frac1{p'}$ and define $\beta := \alpha+ \tfrac{1}{p}-\tfrac1q< 1.$ 
 Let $Q \subseteq \R^d$ be a cube and take $(\omega,\sigma) \in A_{p,q}^\alpha(Q)$ with  $\sigma^{-p'} \in A_{\infty}(Q)$. Then we have
 $$			\nrm{M_{Q}^{\beta}}_{L^p_\sigma(Q) \to L^{q}_\omega(Q)} \lesssim  [\omega,\sigma]_{A_{p,q}^\alpha(Q)}  [\sigma^{-p'}]_{A_\infty(Q)}^{\frac{1}{q}} 
			$$
            The implicit constant depends on $d$ and on $p,q,\alpha$, but is uniform whenever these parameters are bounded away from their endpoints.
	\end{corollary}
    \begin{proof}
        For $f \in L^p_\sigma(Q)$ there exists a sparse collection of cubes $\mc{S}\subseteq \mc{D}(Q)$ such that for a.e. $x \in \R^d$
\[
	M_{Q}^{\beta} f (x) \lesssim_d \sup_{R\in\mc{S}} {\abs{R}^{\beta} \ip{\abs{f}}_{R}}\ind_R(x),  \qquad x \in \R^d,
	\]
which follows from the standard stopping cube argument. Since 
$$
\sup_{R\in\mc{S}} {\abs{R}^{\beta} \ip{\abs{f}}_{R}}\ind_R(x) \leq \mc{A}_{\mc{S}}^{r,\beta}f(x),
$$
for any $r\geq 1$, the claim follows from Proposition \ref{prop:weaksparse}\ref{it:strong} using $r=p$.    
    \end{proof}

    \section{Domination principles}\label{sec:dom}
    A key tool in proving our two-weighted Poincar\'e--Sobolev sandwich is the domination of the oscillation of a function by averages of this oscillation over cubes, allowing us to apply classical Poincar\'e-type inequalities. For the non-critical case of our Poincar\'e--Sobolev inequalities, we will use the following domination principle: For a cube $Q \subseteq \R^d$ and $f\in L^1(Q)$, we have
	\begin{align}
		|f(x)-\ip{f}_{Q}| \lesssim_d \sum_{R\in \mc{D}(Q)}\ip{|f-\ip{f}_R|}_R \ind_R(x),\qquad x\in Q. \label{eq:basicdom}
	\end{align}
	Indeed, let $x\in Q$ and for $j\geq 0$ denote by $Q_j$ the cube in $\mc{D}(Q)$ with $\ell(Q_j)=2^{-j}\ell(Q)$ and $x\in Q_j$. Since
	\begin{align*}
		\absb{\ip{f}_{Q_{j+1}}-\ip{f}_{Q_j}} &\leq \frac{1}{|Q_{j+1}|}\int_{Q_{j+1}}\absb{f(y) - \ip{f}_{Q_j}} \dd y \\
		&\lesssim_{d} \frac{1}{|Q_{j}|}\int_{Q_{j}}\absb{f(y) - \ip{f}_{Q_j}} \dd y = \ip{|f-\ipb{f}_{Q_j}|}_{Q_j},
	\end{align*}
	we have by the triangle inequality,
	\[
		|f(x)-\ip{f}_{Q}| \leq |f(x)-\ip{f}_{Q_n}| + \sum_{j=0}^{n-1}  \absb{ \ip{f}_{Q_{j+1}}-\ip{f}_{Q_j}} \lesssim_d |f(x)-\ip{f}_{Q_n}| + \sum_{j=0}^{n-1}  \ip{|f-\ip{f}_{Q_j}|}_{Q_j}.
	\]
	Now \eqref{eq:basicdom} follows by taking $n\to \infty$ and using the Lebesgue differentiation theorem.

	For the critical cases, we need something similar using a sparse collection of cubes, which is the content of the next lemma.
	\begin{lemma}[{\cite[Lemma 3.1.2]{Hy21}}]\label{lem:sparsedom}
		Let $Q\subseteq \R^d$ be a cube and $f\in L^1(Q)$. There exists a sparse collection of cubes $\mc{S}\subseteq \mc{D}(Q)$ such that
		\begin{align*}
			|f(x)-\ip{f}_{Q}| \lesssim_d \sum_{R\in \mc{S}}\ip{|f-\ip{f}_R|}_R \ind_R(x),\qquad x\in Q.
		\end{align*}
	\end{lemma}

For our third inequality \eqref{eq:introBC}, i.e. the  two-weight Sobolev to Triebel–Lizorkin embedding, we need to dominate the expression appearing in the Triebel--Lizorkin difference norm. We again start with a simple result for the subcritical case. For a cube $R$ and $\gamma>0$, we define $\gamma R$ as the cube with the same center as $R$ and $\ell(\gamma R)=\gamma \ell (R)$. Whenever we integrate a function $g$ defined on a cube $Q$ over a region outside of $Q$, we use its $2\ell(Q)$-periodic extension obtained by even reflection across each face of $Q$\footnote{To be precise, if $Q=\prod_{j=1}^d[a_j,a_j+\ell(Q)]$, define for any point $x\in \R^d$ the point $x^*\in Q$ as follows. For each $j=1,\ldots,d$, let $y_j\in[0,2\ell(Q))$ such that $x_j=a_j+y_j \mod 2\ell(Q)$ and define $x_j^*=a_j+\ell(Q)- |y_j-\ell(Q)|$. Then the extension of $g$ is defined as $\tilde{g}(x)=g(x^*)$.}.
\begin{lemma}\label{lemma:Tldomsubcrit}
		Let $Q \subseteq \R^d$ be a cube, $r \in [1,\infty)$, $s \in (0,1)$ and $f \in F^{s}_{1,r}(Q)$. Then we have for  a.e. $x \in Q$
		\begin{align*}
			\has{\int_Q \frac{|f(x)-f(y)|^r}{|x-y|^{d+sr}}\dd y}^{1/r} &\lesssim_d
    \has{\sum_{R\in \mathcal{D}({Q})} \ell(R)^{-sr}{|f(x)-\ip{f}_{R}|^r}\ind_R(x)}^{1/r},\\
      &\hspace{1cm}+ \has{ \sum_{R\in \mathcal{D}({Q})}\ell(R)^{-sr}   \ipb{|f-\ip{f}_{3R}|}_{r,3R}^r \ind_R(x)}^{1/r}.
\end{align*}
\end{lemma}
\begin{proof}
	Fix $x\in Q$. Let $\cbrace{Q_j}_{j=0}^\infty$ be the sequence of dyadic cubes in $\mc{D}(Q)$ such that $\ell(Q_j) = 2^{-j}\ell(Q)$ and $x \in Q_j$ for all $j\geq 0$. Furthermore, define the annuli
    \begin{align*}
    A_0&:= \cbraceb{y \in Q: 2^{-1} \ell(Q)<|x-y|},\\
        A_j&:= \cbraceb{y \in Q: 2^{-j-1} \ell(Q)<|x-y|\leq 2^{-j}\ell(Q)}, \qquad &&j\geq 1,
    \end{align*}
    and note that $Q = \bigcup_{j=0}^\infty A_j$ up to null sets. Hence, we can decompose 
	\begin{align*}
		\has{\int_{Q} \frac{|f(x)-f(y)|^r}{|x-y|^{d+sr}} \dd y}^{1/r} &=\has{\sum_{j=0}^\infty \int_{A_j} \frac{|f(x)-f(y)|^r}{|x-y|^{d+sr}}\dd y}^{1/r}   \\
        &\lesssim_d  \has{\sum_{j=0}^\infty \ell(Q_j)^{-d-sr} \int_{A_j} {|f(x)-\ip{f}_{Q_j}|^r}\dd y}^{1/r} \\
        &\hspace{1cm} + \has{\sum_{j=0}^\infty \ell(Q_j)^{-d-sr}  \int_{A_j} {|f(y)-\ip{f}_{Q_j}|^r}\dd y}^{1/r} \\ 
         &\lesssim_d  \has{\sum_{j=0}^\infty \ell(Q_j)^{-sr}{|f(x)-\ip{f}_{Q_j}|^r}}^{1/r} \\
        &\hspace{1cm} + \has{\sum_{j=0}^\infty \ell(Q_j)^{-d-sr}  \int_{3 {Q_j}} {|f(y)-\ip{f}_{Q_j}|^r}\dd y}^{1/r}\\
        &=  \has{\sum_{j=0}^\infty \sum_{\substack{R\in \mathcal{D}({Q}):\\ \ell(R)=2^{-j}\ell({Q})}} \mathbf{1}_R(x)\ell(R)^{-sr}{|f(x)-\ip{f}_{R}|^r}}^{1/r} \\
        &\hspace{1cm} + \has{\sum_{j=0}^\infty \sum_{\substack{R\in \mathcal{D}({Q}):\\ \ell(R)=2^{-j}\ell({Q})}}\mathbf{1}_R(x)\ell(R)^{-d-sr}  \int_{3 {R}} {|f(y)-\ip{f}_{R}|^r}\dd y}^{1/r}.
	\end{align*}
    To finish the proof, note that for each $R$ we have by H\"older,
    \begin{align*}
             \has{\int_{3 {R}} {|f(y)-\ip{f}_{R}|^r}\dd y}^{1/r} &\leq  \has{\int_{3 {R}} {|f(y)-\ip{f}_{3R}|^r}\dd y}^{1/r} \\
         &\hspace{1cm}+\has{\int_{3R} {|\ip{f}_{R}-\ip{f}_{3R}|^r}\dd y}^{1/r}\\&\lesssim_d \has{\int_{3R} {|f(y)-\ip{f}_{3R}|^r}\dd y}^{1/r}.  \qedhere
    \end{align*}
\end{proof}

    For the critical case, we will again need a sparse version of this domination result, which is new and interesting in its own right.
	\begin{theorem}\label{thm:fracsparsedom}
Let $Q \subseteq \R^d$ be a cube, $r \in [1,\infty)$, $s \in (0,1)$ and $f \in F^{s}_{1,r}(Q)$. For $R \subseteq Q$ define         
        $$
f_{R}^{s,r}(x):=\has{\int_R \frac{|f(x)-f(y)|^r}{|x-y|^{d+sr}}\dd y}^{1/r}, \qquad x \in R.
$$
Then there exists a sparse collection $\mathcal{S}\subseteq \mathcal{D}(Q)$ such that for $x \in Q$
		\begin{align*}
			f_{Q}^{s,r}(x) \lesssim_d \sum_{R\in\mathcal{S}} \has{ \ipb{f_{3R}^{s,r}}_R  + \frac1{s^{1+1/r}\cdot \ell(R)^s}\cdot {\ipb{|f-\ipb{f}_R|}_R} }\mathbf{1}_R(x) .
		\end{align*}
	\end{theorem}
	\begin{proof}
		We will construct $\mathcal{S}$ iteratively. Set $\mathcal{S}_0=\{{Q}\}$ and suppose that $\mathcal{S}_n$ has been constructed. Fix $R\in\mathcal{S}_n$ and define $\mathcal{S}_R\subseteq \mathcal{D}(R)$ to be the set of maximal cubes $S \in \mc{D}(R)$ such that at least one of the following conditions is true
		\begin{align}
			\ipb{f_{3R}^{s,r}}_S &> 4 \,\ipb{f_{3R}^{s,r}}_R  ,\label{eq:sparscond1}\\
			\ipb{|f-\ipb{f}_R|}_S &> 4 \,\ipb{|f-\ipb{f}_R|}_R \label{eq:sparscond2}.
		\end{align}
		Define $\mathcal{S}_{n+1} = \bigcup_{R\in \mathcal{S}_{n}}\mathcal{S}_R$ and set $\mathcal{S}=\bigcup_{n=0}^\infty \mathcal{S}_n$.
We claim that $\mc{S}$ is sparse. Indeed, for any $R \in \mc{S}$ use
        $$
E_R:= R\setminus \bigcup_{S\in \mathcal{S}_R} S,
        $$
        which are pairwise disjoint by construction. Moreover, denoting the set of all $S\in\mathcal{S}_R$ that satisfy \eqref{eq:sparscond1} and \eqref{eq:sparscond2} by $\mathcal{S}_R^1$ and $\mathcal{S}_R^2$ respectively, we get from \eqref{eq:sparscond1} that for $S\in \mathcal{S}_R^1$,
		\begin{align*}
			|S| \leq |R| \frac{\int_{S}f_{3R}^{s,r}(x)\dx}{4\int_{R}f_{3R}^{s,r}(x)\dx},
		\end{align*}
		and consequently
		\begin{align*}
			\sum_{S\in \mathcal{S}_R^1} |S| \leq \tfrac14 |R|.
		\end{align*}
		Similarly, we get
		\begin{align*}
			\sum_{S\in \mathcal{S}_R^2} |S| \leq \tfrac14 |R|,
		\end{align*}
		and therefore
		\begin{align*}
			\abs{E_R} = \abss{R\setminus \bigcup_{S\in \mathcal{S}_R} S} \geq \abs{R} - \sum_{S\in \mathcal{S}_R^1} |S| - \sum_{S\in \mathcal{S}_R^2} |S| \geq \tfrac12 |R|,
		\end{align*}
		 so $\mathcal{S}$ is indeed sparse.

        \medskip

We want to show that for all $R\in \mathcal{S}$ and  a.e. $x \in R$ we have
		\begin{align}
        \begin{aligned}
        f_{3R}^{s,r}(x)\mathbf{1}_R(x) \leq c_d \cdot  \ip{f_{3R}^{s,r}}_R \mathbf{1}_R(x)  &+c_d \cdot {s^{-\frac1r}} \cdot \sum_{S\in \mathcal{S}_R}\ell(S)^{-s} |f(x)-\ip{f}_S|\mathbf{1}_S(x)\\
            &+\sum_{S\in \mathcal{S}_R} f_{3S}^{s,r}(x)\mathbf{1}_S(x)
        \end{aligned}
			 \label{eq:sparsestimate1}
		\end{align}
		and
		\begin{align}
			|f(x)-\ip{f}_R|\mathbf{1}_R(x) &\leq  c_d \cdot \ipb{|f-\ip{f}_R|}_R\mathbf{1}_R(x)+\sum_{S\in \mathcal{S}_R}|f(x)-\ip{f}_S|\mathbf{1}_S(x) , \label{eq:sparsestimate2}
		\end{align}
		where $c_d$ is some dimensional constant. This would imply the statement of this theorem. Indeed, define 
		\[
		N_0 = \bigcap_{n=0}^\infty \bigcup_{R\in\mathcal{S}_n} R = Q \setminus \bigcup_{R \in \mc{S}} E_R.
		\]
Then we have by sparsity
		\[
		|N_0| = \lim\limits_{n\to\infty} \absb{ \bigcup_{R\in\mathcal{S}_n} R}\leq \lim\limits_{n\to\infty}(\tfrac12)^n |{Q}| = 0.
		\]
Now, for every $R \in \mc{S}$ let $N_R\subseteq R$ be a set such that $|N_R|=0$ and  \eqref{eq:sparsestimate1} and \eqref{eq:sparsestimate2} hold for all $x \in R \setminus N_R$. Set $$N:= N_0\cup \bigcup_{R \in \mc{S}}N_R,$$ which is a set of measure zero. Fix $x\in {Q}\setminus {N}$, then by construction there exists an integer $n$ and cubes $S_j\in\mathcal{S}_j$  such that $x\in S_j$ for $j=0,\ldots,n$ and $x\in E_{S_n}$. Since $x \notin S$ for all $S \in \mathcal{S}_{n+1}$, applying \eqref{eq:sparsestimate1} $(n+1)$-times gives
		\begin{align*}
			f_{3{Q}}^{s,r}(x) \lesssim_d \sum_{j=0}^n  \ipb{f_{3S_j}^{s,r}}_{S_j} + {s^{-\frac1r}}\cdot  \sum_{j=1}^n  \ell(S_j)^{-s}\abs{f(x)-\ip{f}_{S_j}}.
		\end{align*}
		Then, for each $j$ applying \eqref{eq:sparsestimate2} $(n+1-j)$-times gives
		\[
		\absb{f(x)-\ip{f}_{S_j}} \lesssim  \sum_{m=j}^{n}\ipb{\absb{f-\ip{f}_{S_m}}}_{S_m}.
		\]
Combining these  estimates gives
		\begin{align*}
			f_Q^{s,r}(x)\leq f_{3{Q}}^{s,r}(x)&\lesssim_d \sum_{j=0}^n  \ip{f_{3S_j}^{s,r}}_{S_j} + {s^{-\frac1r}} \cdot \sum_{j=1}^n \ell(S_j)^{-s} \sum_{m=j}^{n}\ipb{\absb{f-\ip{f}_{S_m}}}_{S_m}\\
            &\lesssim_d \sum_{j=0}^n  \ip{f_{3S_j}^{s,r}}_{S_j} + \beta\cdot{s^{-\frac1r}} \cdot\sum_{m=1}^n  \ell(S_m)^{-s}\ipb{\absb{f-\ip{f}_{S_m}}}_{S_m},
		\end{align*}
		where
		\[
		\beta= \sum_{j=1}^m \frac{\ell(S_m)^{s}}{\ell(S_j)^s} \leq \sum_{j=1}^m 2^{(j-m)s} \lesssim \frac{1}{s},
		\]
		proving the theorem.

        \medskip
		
		It remains to prove \eqref{eq:sparsestimate1} and \eqref{eq:sparsestimate2}, for which we fix $R\in\mathcal{S}$. By the Lebesgue differentiation theorem, we have for  a.e. $x \in E_R$ that 
		\begin{align*}
			f_{3R}^{s,r}(x) &\leq 4 \,\ipb{f_{3R}^{s,r}}_R. \\
			|f(x)-\ipb{f}_R| & \leq  4 \,\ipb{\abs{f-\ipb{f}_R}}_R,
		\end{align*}
so \eqref{eq:sparsestimate1} and \eqref{eq:sparsestimate2} hold.    
		If $x \in  R \setminus E_R$, then $x \in S$ for some $S\in \mathcal{S}_R$. Then we have 
		\begin{align*}
			|f(x)-\ip{f}_R| - |f(x)-\ip{f}_S| &\leq  |\ip{f}_S-\ip{f}_R| \\
			&\leq  \frac{1}{|S|} \int_S |f(y)-\ipb{f}_R|\dd y \\
			&\lesssim_d \frac{1}{\absb{\widehat{S}}} \int_{\widehat{S}} |f(y)-\ipb{f}_R|\dd y\\&\leq 
            4\cdot \ipb{|f-\ip{f}_R|}_R
		\end{align*}
		where $\widehat{S}$ is the dyadic parent of $S$ (i.e. $\widehat{S}\in \mc{D}(Q)$ such that $\ell(\widehat{S})=2\ell(S)$ and $S\subseteq \widehat{S}$) and the final step follows by the maximality of $S$ satisfying \eqref{eq:sparscond2}. This proves \eqref{eq:sparsestimate2}. For \eqref{eq:sparsestimate1}  we estimate
		\begin{align*}
			f^{s,r}_{3R}(x)-f^{s,r}_{3S}(x) &\leq \has{\int_{3R\backslash 3S} \frac{|f(x)-f(y)|^r}{|x-y|^{d+sr}}\dd y }^{1/r} \\
			&\leq \has{\int_{3R\backslash 3S} \frac{|f(x)-\ip{f}_S|^r}{|x-y|^{d+sr}}\dd y }^{1/r} + \has{\int_{3R\backslash 3S} \frac{|\ip{f}_S-f(y)|^r}{|x-y|^{d+sr}}\dd y }^{1/r}.
		\end{align*}
		For the first term, note that $|x-y|\geq \ell(S)$, so that
		\begin{align*}
		    \has{\int_{3R\backslash 3S} \frac{|f(x)-\ip{f}_S|^r}{|x-y|^{d+sr}}\dd y }^{1/r}  &\lesssim_d  \has{\int_{\abs{x-y}\geq \ell(S)} \frac{1}{|x-y|^{d+sr}}\dd y }^{1/r} \cdot |f(x)-\ip{f}_S|\\ 
            &\lesssim_d {s^{-\frac1r}} \cdot \ell(S)^{-s}|f(x)-\ip{f}_S|.
		\end{align*}
		For the second term, we have
		\[
		\frac{|\ip{f}_S-f(y)|}{|x-y|^{d/r+s}} \leq \frac{1}{|S|}\int_S \frac{|f(z)-f(y)|}{|x-y|^{d/r+s}} \dd z \lesssim_d \frac{1}{|S|}\int_S \frac{|f(z)-f(y)|}{|z-y|^{d/r+s}} \dd z,
		\]
		where we used that $|x-y| \geq \tfrac1{2\sqrt{d}} |y-z|$ since $x,z\in S$ and $y\in 3R\setminus 3S$. Therefore, we have 
		\begin{align*}
			\left(\int_{3R\backslash 3S} \frac{|\ip{f}_S-f(y)|^r}{|x-y|^{d+sr}}\dd y \right)^{1/r} &\lesssim_d \frac{1}{|S|}\left(\int_{3R}\left(\int_S \frac{|f(z)-f(y)|}{|z-y|^{d/r+s}}\dd z\right)^r \dd y \right)^{1/r} \\
			&\leq \frac{1}{|S|}\int_S\left(\int_{3R} \frac{|f(z)-f(y)|^r}{|z-y|^{d+sr}}\dd y\right)^{1/r} \dd z \\
			&\lesssim_d \ipb{f^{s,r}_{3R}}_{\widehat{S}}\leq 4 \ipb{f^{s,r}_{3R}}_{R},
		\end{align*}
		where $\widehat{S}$ is again the dyadic parent of $S$ and we used the maximality of $S$ satisfying \eqref{eq:sparscond1} in the final step. This proves \eqref{eq:sparsestimate1} and thus finishes the proof.
	\end{proof}

\section{Two-weight Poincar\'e--Sobolev inequality} \label{sec:AC}
We start by studying a two-weight version of the classical Poincar\'e--Sobolev inequality, i.e.
$$
\nrm{f-\ip{f}_{Q}}_{L^q_\omega(Q)} \lesssim \nrmb{\abs{\nabla f}}_{L^p_\sigma(Q)}.
$$
under appropriate conditions on $p,q,\sigma,\omega$. It is worth noting that there are several natural ways to normalize the oscillation on the left-hand side. Besides subtracting the Lebesgue average $\ip{f}_{Q}$, one could also subtract the weighted average
\[
\ip{f}_{Q}^{\omega^q}:=\frac{1}{\omega^q(Q)}\int_Q f\,\omega^q,
\]
or, more intrinsically, take the infimum over all constants $c\in \C$. Our formulation with the unweighted average $\ip{f}_{Q}$ is in general the strongest one. Indeed, we have
\[
\inf_{c\in \C}\, \nrmb{f-c}_{L^q_\omega(Q)} \leq \nrmb{f-\ip{f}_{Q}^{\omega^q}}_{L^q_\omega(Q)}
\leq 2 \,\nrm{f-\ip{f}_{Q}}_{L^q_\omega(Q)}.
\]
Conversely, in the one-weight Muckenhoupt setting these formulations are equivalent up to constants depending only on the weight characteristic. More precisely, if  $\omega^q \in A_q$, then for every $c\in \C$ one has
\[
|\ip{f}_{Q}-c|\,\omega^q(Q)^{1/q}
\lesssim [\omega^q]_{A_q}^{1/q}\nrm{f-c}_{L^q_\omega(Q)}
\]
and hence
\begin{equation}\label{eq:avgbelowinf}
\nrm{f-\ip{f}_{Q}}_{L^q_\omega(Q)}
\lesssim
[\omega^q]_{A_q}^{1/q}
\inf_{c\in \C}\, \nrm{f-c}_{L^q_\omega(Q)}.
\end{equation} Therefore, in the one-weight $A_p$-setting the three formulations are equivalent, whereas in the general two-weight setting our choice using $\ip{f}_{Q}$ yields the strongest results.

Our claims to novelty in this section are rather mild. Instead, we would like to emphasize that our proofs unify and significantly simplify the existing literature. For a comparison to the literature, see Subsection \ref{sec:ACliterature}.

\subsection{Main result}
Our main two-weight Poincar\'e--Sobolev inequality reads as follows.

	\begin{theorem}\label{theorem:IleqIII}
	Let $Q\subseteq \R^d$ be a cube, let $1\leq p\leq q <\infty$, $0\leq \alpha < \tfrac1q+\tfrac1{p'}$, $(\omega,\sigma)\in \Apq$ and assume
	\[
	\varepsilon:=\tfrac1d - (\tfrac1p - \tfrac1q)-\alpha \geq 0.
	\]
    Then we have the following assertions for $f \in W^{1,p}_\sigma(Q)$.
    \begin{enumerate}[(i)]
		\item\label{it:diagpqcase1} \emph{(Subcritical case)} If $\varepsilon>0$, we have
		\begin{align*}
			\nrm{f-\ip{f}_{Q}}_{L^q_\omega(Q)} \lesssim_d  [\omega,\sigma]_{A_{p,q}^\alpha(Q)}\cdot \frac{\abs{Q}^{\varepsilon}}{\varepsilon}  \nrmb{\abs{\nabla f}}_{L^p_\sigma(Q)}.
		\end{align*}
		\item\label{it:diagpqcase2}  \emph{(Critical case)} If $\varepsilon=0$ and we additionally assume  $\omega^q\in A_\infty(Q)$, then 
		\begin{align*}
			{\nrm{f-\ip{f}_{Q}}_{L^q_\omega(Q)}} \lesssim_{p,q,d}  \nrmb{\abs{\nabla f}}_{L^p_\sigma({Q})} \cdot [\omega,\sigma]_{A_{p,q}^\alpha({Q})}  \cdot \begin{cases}
				[\omega^q]_{A_\infty(Q)}^{\frac{1}{p'}}\qquad & 1<p<q,\\
				[\omega^q]_{A_\infty(Q)} & \text{otherwise}. 
			\end{cases}
   \end{align*}
	\end{enumerate}
	\end{theorem}
	Before we turn to the proof, we will need some preparations. First of all, we need the well-known $(1,1)$-Poincar\'e inequality on a cube, i.e. for a cube $Q\subseteq \R^d$ and $f \in W^{1,1}(Q)$ we have
	\begin{align}
		\nrm{f-\ip{f}_{Q}}_{L^1(Q)} \lesssim_d \ell(Q)\cdot \nrmb{\abs{\nabla f}}_{L^1(Q)}. \label{eq:11poincare}
	\end{align}
In fact, this inequality holds for any convex set, replacing $\ell(Q)$ by the diameter of the set.
    
	To prove the subcritical case, we will also need the following inequality. We formulate a slightly more general result, which we  will also  use  in Sections \ref{sec:AB} and \ref{sec:BC}.    
    
	\begin{lemma}\label{lem:QtoQ0}
		Let $Q\subseteq \R^d$ be a cube, $1\leq p \leq q <\infty$, $\gamma \in [1,\infty)$, $\varepsilon>0$ and let $\omega,\sigma$ be weights. For $g \in L^p_{\sigma}(Q)$ we have 
        \begin{align*}
            \has{\int_{Q} \has{\sum_{R \in \mc{D}(Q)} \frac{\abs{R}^{\varepsilon}}{\omega^q(R)^{1/q}} \nrm{g}_{L^p_\sigma(\gamma R)} \ind_R}^{q} \omega^q\dx}^{1/q} \lesssim_{\gamma,d} \frac{|Q|^\varepsilon}{\varepsilon} \nrm{g}_{L^p_\sigma(Q)}.
        \end{align*}
	\end{lemma}
	\begin{proof}
Note that the collection of all $R\in \mc{D}(Q)$ with the same side length is pairwise disjoint and the cubes $\gamma R$ have bounded overlap. Therefore, using $\ell^p\hookrightarrow \ell^q$ in the final step, we can estimate
		\begin{align*}
			\has{\int_{Q} &\has{\sum_{R \in \mc{D}(Q)} \frac{\abs{R}^{\varepsilon}}{\omega^q(R)^{1/q}} \nrm{g}_{L^p_\sigma(\gamma R)} \ind_R}^{q} \omega^q\dx}^{1/q}\\
			&\leq \abs{Q}^{\varepsilon}\sum_{j=0}^\infty 2^{-jd\varepsilon} \has{\int_{Q} \has{\sum_{\substack{R \in \mc{D}(Q):\\ \ell(R) = 2^{-j}\ell(Q)}}  \frac{1}{\omega^q(R)^{1/q}} \nrm{g}_{L^p_\sigma(\gamma R)} \ind_R}^q \omega^q\dx}^{1/q}\\
            &= \abs{Q}^{\varepsilon}\sum_{j=0}^\infty 2^{-jd\varepsilon} \has{\int_{Q} \sum_{\substack{R \in \mc{D}(Q):\\ \ell(R) = 2^{-j}\ell(Q)}}  \frac{1}{\omega^q(R)} \nrm{g}_{L^p_\sigma(\gamma R)}^q \ind_R \omega^q\dx}^{1/q}\\
            &= \abs{Q}^{\varepsilon}\sum_{j=0}^\infty 2^{-jd\varepsilon} \has{ \sum_{\substack{R \in \mc{D}(Q):\\ \ell(R) = 2^{-j}\ell(Q)}}  \nrm{g}_{L^p_\sigma(\gamma R)}^q}^{1/q} \\
            &\lesssim_\gamma \abs{Q}^{\varepsilon} \cdot \nrm{g}_{L^p_\sigma(Q)} \cdot \sum_{j=0}^\infty 2^{-jd\varepsilon}.
		\end{align*}
Since
		$
		\sum_{j=0}^\infty 2^{-jd\varepsilon} = \frac{1}{1-2^{-d\varepsilon}} \lesssim \frac{1}{\varepsilon},
		$
		 the lemma follows.
	\end{proof}

	Now we are ready to prove the main theorem of this section.
	\begin{proof}[Proof of Theorem \ref{theorem:IleqIII}]
    Let $f \in W^{1,p}_\sigma(Q)$. For both cases \ref{it:diagpqcase1} and \ref{it:diagpqcase2} we will use the $(1,1)$-Poincar\'e inequality \eqref{eq:11poincare}. This requires $f \in W^{1,1}(Q)$, which follows from H\"older's inequality as $[\omega,\sigma]_{\Apq}<\infty$ and therefore $\sigma^{-1} \in L^{p'}(Q)$. 
    
	For \ref{it:diagpqcase1}, let $R \in \mc{D}(Q)$. By the $(1,1)$-Poincar\'e inequality \eqref{eq:11poincare} and H\"older's inequality, we have
	\begin{align}
    \begin{split}
		\frac{1}{\abs{R}} \int_R \abs{f-\ip{f}_R} \dx &\lesssim_d \frac{\ell(R)}{\abs{R}} \int_R \abs{\nabla f}\dx
		\\&\leq \frac{1}{\abs{R}^{1-\frac1d}} \has{\int_R \abs{\nabla f}^p  \sigma^p\dx}^{1/p} \has{\int_R \sigma^{-p'}\dx}^{1/p'}\\
		&\leq [\omega,\sigma]_{A_{p,q}^\alpha(Q)} \frac{\abs{R}^{\varepsilon}}{\omega^q(R)^{1/q}} \nrmb{\abs{\nabla f}}_{L^p_\sigma(R)}.\end{split}\label{eq:avgcontrolAC}
	\end{align}
    Using the domination principle from \eqref{eq:basicdom}, we obtain
	\begin{align*}
		\has{\int_{Q}& \abs{f-\ip{f}_{Q}}^{q} \omega^q\dx}^{1/q} \\
		&\leq\has{\int_{Q} \has{\sum_{R \in \mc{D}({Q})} \ip{\abs{f- \ip{f}_R}}_{R} \ind_R}^{q} \omega^q\dx}^{1/q}\\
		&\lesssim_d [\omega,\sigma]_{A_{p,q}^\alpha({Q})}  \has{\int_{Q} \has{\sum_{R \in \mc{D}({Q})} \frac{\abs{R}^{\varepsilon}}{\omega^q(R)^{1/q}} \nrmb{\abs{\nabla f}}_{L^p_\sigma(R)} \ind_R}^{q} \omega^q\dx}^{1/q}.
    \end{align*} 
    The claimed estimate \ref{it:diagpqcase1} now follows from Lemma \ref{lem:QtoQ0}.

    \medskip
    
	For \ref{it:diagpqcase2}, we assume $\varepsilon=0$ and therefore
	$
		\frac1d =\alpha + \frac1p - \frac1q.
	$
	Let $\beta= \frac1d = \alpha + \frac1p - \frac1q \in (0,1)$ and note that by the $(1,1)$-Poincar\'e inequality \eqref{eq:11poincare} we have
	\begin{align*}
		\frac{1}{\abs{R}} \int_R \abs{f-\ip{f}_R} \dx &\lesssim_d \frac{\ell(R)}{\abs{R}} \int_R \abs{\nabla f}\dx = |R|^{\beta} \ipb{\abs{\nabla f}}_{R}.
	\end{align*}
	By Lemma \ref{lem:sparsedom}, there exists a sparse collection of cubes $\mc{S}\subseteq \mc{D}({Q})$ such that
	\begin{align*}
		\abs{f(x)-\ip{f}_{{Q}}} \lesssim_d \sum_{R \in \mc{S}} \ipb{\abs{f - \ip{f}_R}}_{R} \ind_R, \qquad x \in {Q}.
	\end{align*}
	Combining above two inequalities, we obtain for a.e. $x \in Q$
	\begin{align*}
		\abs{f(x)-\ip{f}_{Q}} &\lesssim_d \sum_{R \in \mc{S}} \ipb{\abs{f- \ip{f}_R}}_{R} \ind_R(x) \\
		&\lesssim_d  \sum_{R \in \mc{S}} |R|^{\beta} \ipb{\abs{\nabla f}}_{R} \ind_R(x) \\
		&= \mathcal{A}_\mathcal{S}^{1,\beta} (|\nabla f|)(x).
	\end{align*}
	Therefore, Proposition \ref{prop:weaksparse}\ref{it:weak}  yields
	\begin{align*}
		{\nrmb{f-\ip{f}_{Q} }_{L^{q,\infty}_\omega({Q})}} \lesssim_{p,q,d}  \nrmb{\abs{\nabla f}}_{L^p_\sigma({Q})} \cdot [\omega,\sigma]_{A_{p,q}^\alpha({Q})}  \cdot \begin{cases}
				[\omega^q]_{A_\infty(Q)}^{\frac{1}{p'}}\qquad & 1<p<q,\\
				[\omega^q]_{A_\infty(Q)} & \text{otherwise}. 
			\end{cases}
	\end{align*}
	The result now follows from the weak implies strong principle, see Proposition \ref{prop:weaktostrongclassic}, by using \eqref{eq:avgcontrolAC} with $|R|=|Q|$, $\varepsilon=0$ and noting that $[\omega^q]_{A_\infty(Q)}\geq 1$.
	\end{proof}

    \begin{remark}\label{rem:remafterAC}~
        \begin{enumerate}[(i)]
        \item \label{it:ACsimple} As already explained in the introduction, the proof of the subcritical case in \ref{theorem:IleqIII}\ref{it:diagpqcase1} is completely elementary, only using the following results:
    \begin{itemize}
        \item The $(1,1)$-Poincar\'e inequality \eqref{eq:11poincare}.
        \item The domination principle \eqref{eq:basicdom}, which is a direct consequence of the Lebesgue differentiation theorem.
        \item The dyadic summing principle in Lemma \ref{lem:QtoQ0}.
    \end{itemize}
            \item In the setting of Theorem \ref{theorem:IleqIII}\ref{it:diagpqcase2} with $p=q>1$, if we additionally assume $\sigma^{-p'}\in A_\infty(Q)$, we can also show that 
        \begin{align*}
			{\nrm{f-\ip{f}_{Q}}_{L^q_\omega(Q)}} \lesssim_{p,q,\alpha,d}  \nrmb{\abs{\nabla f}}_{L^p_\sigma({Q})} \cdot [\omega,\sigma]_{A_{p,p}^\alpha({Q})}  \cdot
				[\omega^p]_{A_\infty(Q)}^{\frac{1}{p'}} \cdot  [\sigma^{-p'}]_{A_\infty(Q)}^{\frac{1}{p}}.
   \end{align*}
   This follows by using Proposition \ref{prop:weaksparse}\ref{it:strong} instead of Proposition \ref{prop:weaksparse}\ref{it:weak}.
   While qualitatively weaker than Theorem \ref{theorem:IleqIII}\ref{it:diagpqcase2}  due to the additional assumption on $\sigma$, this can give a quantitatively  stronger estimate if
   $
[\sigma^{-p'}]_{A_\infty(Q)} \leq [\omega^p]_{A_\infty(Q)}.
   $
        \end{enumerate}
    \end{remark}

\medskip

Let us specify Theorem \ref{theorem:IleqIII} to the one-weight case using Lemma \ref{lem:ApApq}. Note that the three cases in the constant below are not mutually exclusive. Hence, for some choices of parameters more than one case may apply and one may choose whichever is smallest.
        \begin{corollary}\label{cor:AleqConeweight}
Let $Q\subseteq \R^d$ be a cube, let $1\leq u\leq p\leq q <\infty$, $w\in A_{u}$ and assume
	\[
	\varepsilon:=\tfrac1{du} - (\tfrac1p - \tfrac1q)\geq 0.
	\]
    Then we have for $f \in W^{1,p}_{w^{1/p}}(Q)$
		\begin{align*}
			\has{\frac{1}{w(Q)}\int_{Q}\abs{f-\ip{f}_{Q}}^qw\dd x}^\frac1q \lesssim_{p,q,d} \ell(Q)  &\cdot \has{\frac{1}{w(Q)}\int_{Q}{\abs{\nabla f}}^pw \dd x}^{\frac1p}.\\
            &\cdot [w]_{A_p}^{\frac{1}{p}} [w]_{A_u}^{\frac{1}{p}-\frac1q} \cdot \begin{cases}
			    \tfrac{1}{\varepsilon}, \quad &\varepsilon>0,\\
                [w]_{A_\infty}^{\frac{1}{p'}}, \quad &p>1,\\
                [w]_{A_\infty}, \qquad & p=1.
			\end{cases}
		\end{align*}
    \end{corollary}
\begin{proof}
    If $\varepsilon=0$, the claim follows by using Theorem \ref{theorem:IleqIII}\ref{it:diagpqcase2} with $\alpha  = (u-1) (\frac1p-\frac1q)$ and Lemma \ref{lem:ApApq}. If $p>1$ and $\tfrac{1}{\varepsilon}\leq du \cdot  [w]_{A_\infty}^{1/p'}$, the claim follows by using Theorem \ref{theorem:IleqIII}\ref{it:diagpqcase1} with $\tilde{\varepsilon}:=\varepsilon u$, respectively, again in combination with Lemma \ref{lem:ApApq}.  If $\tfrac{1}{\varepsilon}\geq du \cdot [w]_{A_\infty(Q)}^{1/{p'}}$, then $p<q$, so the claim follows by using Theorem \ref{theorem:IleqIII}\ref{it:diagpqcase2} with $\alpha := (u-1)(\tfrac1p - \tfrac1q) + \varepsilon u$, noting that by Lemma \ref{lem:propertiesMuckenhoupt}\ref{it:weightalphabeta}
    $$
    [w^{1/q},w^{1/p}]_{A_{p,q}^{\alpha}(Q)} \leq \abs{Q}^{\varepsilon u}[w^{1/q},w^{1/p}]_{A_{p,q}^{\alpha-\varepsilon u}(Q) },
     $$
    and then using Lemma \ref{lem:ApApq} with $\tilde{\alpha} := \alpha-\varepsilon u =(u-1)(\tfrac1p - \tfrac1q)$. The proof when $p=1$ is analogous.
\end{proof}

    \subsection{Comparison to the literature}\label{sec:ACliterature}
In this subsection, we compare Theorem \ref{theorem:IleqIII} and Corollary \ref{cor:AleqConeweight} with earlier results in the literature. We start with a comparison of Theorem \ref{theorem:IleqIII}  to previous results under a two-weight $(p,q)$-Muckenhoupt condition, possibly supplemented by $A_\infty$-assumptions on one or both of the weights. 
\begin{enumerate}[(i)]
\item In comparison with the general proof strategy for Poincaré--Sobolev inequalities through self-improving phenomena in e.g.  \cite{FPW98,LP05,MP98,PR19}, our subcritical case corresponds to \(SD_p^s(w)\)-condition, while the critical case corresponds to the \(D_p(w)\)-condition. The sharper estimates obtained here rely on exploiting the full 
$L^p$-structure of the inequalities, rather than deducing them from an abstract general theory.
    \item The subcritical case in Theorem \ref{theorem:IleqIII}\ref{it:diagpqcase1} for $p=q$, $\alpha=0$, additionally assuming $\omega^q \in A_\infty$, was obtained in \cite[Corollary 1.8]{PR19}. The $A_\infty$-condition was subsequently removed in \cite[Theorem 5.3]{LLO21}. In addition to allowing $p\neq q$, and $\alpha \neq 0$, our  proof of Theorem \ref{theorem:IleqIII}\ref{it:diagpqcase1} is also much more elementary than \cite{LLO21,PR19}, see  Remark \ref{rem:remafterAC}\ref{it:ACsimple}. Consequently, our approach can, e.g., be extended to the multi-parameter setting, where sparse domination techniques may not be available \cite{BCOR19}, see Remark \ref{remark:AleqCRect} below.
    \item Qualitative versions of Theorem \ref{theorem:IleqIII}\ref{it:diagpqcase2} go back to e.g. \cite{CW85, CW92, Chu93}, with conditions on the weights on a dilate of $Q$. 
    Assuming only conditions on the weights formulated on $Q$ itself, Theorem \ref{theorem:IleqIII}\ref{it:diagpqcase2}  with the additional assumption that $\sigma^{-p'} \in A_\infty$ and $p>1$ and without explicit dependence on the weight characteristics  was obtained in \cite[Theorem 5.4]{KV21} (see also \cite[Theorem 9.21]{KLV21}). These additional assumptions are absent in Theorem \ref{theorem:IleqIII}\ref{it:diagpqcase2}.
    \item Using the local subrepresentation formula
    \begin{equation}\label{eq:localsub}
\abs{f(x)-\ip{f}_{Q}} \lesssim_d \int_{Q} \frac{\abs{\nabla f(y)}}{|x-y|^{d-1}} \dd y=: I_1(\ind_Q\abs{\nabla f})(x), \qquad x \in Q
    \end{equation}
    for $f \in W^{1,1}(\R^d)$, one can deduce two-weight Poincar\'e inequalities from two-weight estimates for the fractional integral operator $I_1$, see, e.g., \cite{SW92}. Sharp weak-type two-weight estimates for $I_1$ can be found in \cite[Theorem 2.2]{CM13} (see \cite{LMPT10} for the one-weight case). Combined with the truncation method, this provides an alternative proof of Theorem \ref{theorem:IleqIII}\ref{it:diagpqcase2} in the case $p<q$. 
    \item In  \cite{DD08} a variant of the approach via the local subrepresentation formula \eqref{eq:localsub} was used to obtain two-weighted Poincar\'e--Sobolev inequalities with an additional factor measuring the distance to the boundary of $Q$. Their main results can also be shown using  Theorem \ref{theorem:IleqIII}\ref{it:diagpqcase2} and a Whitney decomposition.
    \item \label{it:Mwcomp} In \cite[Theorem 1.21]{PR19} it was shown that 
\begin{align*}
			{\nrm{f-\ip{f}_{Q}^{\omega^q}}_{L^q_\omega(Q)}} \lesssim  \nrmb{\abs{\nabla f}}_{L^p_\sigma({Q})} 
   \end{align*}
   for $\frac1p = \frac1q+\frac1n$,  and where $\omega = w^{1/q}$ for a general weight $w$ and $\sigma := M(w\ind_Q)^{\frac{1}{n'}}\cdot w^{-\frac1{p'}}$. A simple calculation shows that $(\omega,\sigma) \in A_{p,q}^0(Q)$, so this result fits in the framework of Theorem \ref{theorem:IleqIII}\ref{it:diagpqcase2}. However, the results are incomparable, as we would require an $A_\infty$-condition on $w$.    
\end{enumerate}
Next, we compare the one-weight estimate in Corollary \ref{cor:AleqConeweight} with existing results. 
\begin{enumerate}[(i)]\setcounter{enumi}{6}
    \item Qualitatively, the result in Corollary \ref{cor:AleqConeweight} goes back to \cite[Theorem 1.5]{FKS82}, see also \cite[Chapter 15]{HKM06}. A quantitative version  of this result in terms of weight characteristics was obtained in the subcritical regime for a specific $\varepsilon>0$ in \cite[Corollary 1.13]{PR19} and in the critical regime $\varepsilon=0$ in \cite[Corollary 1.15]{PR19}. Our quantitative weight dependence is identical in the subcritical regime when taking \begin{align*}
        \frac1p-\frac1q = \frac 1{d}\cdot\frac{1}{u +\log[w]_{A_u}} \qquad &\Longrightarrow \qquad [w]_{A_u}^{\frac1p-\frac1q}\lesssim 1 \quad \text{and}\quad \varepsilon = \frac1d\cdot\frac{ \log[w]_{A_u}}{u(u+\log[w]_{A_u})}\\
        &\Longrightarrow \qquad [w]_{A_u}^{\frac1p-\frac1q}\cdot \min\cbraceb{\tfrac{1}{\varepsilon},[w]_{A_\infty}}\lesssim_{u,d} 1.
    \end{align*}
    The results in the critical regime are quantitatively incomparable.
    \item   
    The subcritical case for a specific $\varepsilon>0$ and the critical case $\varepsilon=0$ with $u=1$ of Corollary \ref{cor:AleqConeweight} were recently obtained in \cite[Theorem 2.4]{Cla25}, using the local subrepresentation formula \eqref{eq:localsub} and one-weight estimates for $I_1$. The dependence on the weight characteristic in \cite{Cla25} is sharp and smaller than the dependence in Corollary \ref{cor:AleqConeweight}. The suboptimality of Corollary \ref{cor:AleqConeweight} in this case stems from the fact that it is derived as a specialization of our general two-weight theory, rather than by means of arguments tailored to the one-weight setting.
\end{enumerate}

\begin{remark}\label{remark:AleqCRect}
   The proof ingredients for Theorem \ref{theorem:IleqIII}\ref{it:diagpqcase1}, i.e. \eqref{eq:11poincare}, \eqref{eq:basicdom}, and Lemma \ref{lem:QtoQ0}, are also 
 available for a rectangle $R\subseteq \R^d$ and its dyadic subrectangles. Hence, using suitably adapted rectangular Muckenhoupt weight classes, Theorem \ref{theorem:IleqIII}\ref{it:diagpqcase1} and Corollary \ref{cor:AleqConeweight} with $\varepsilon>0$ also hold for $R$, which yields a simple proof of \cite[Corollary 2.7]{CMPR23} in the case $m=1$. We leave the details to the interested reader.
\end{remark}

\section{Two-weight  fractional Poincar\'e--Sobolev inequality} \label{sec:AB}	
We now turn to the two-weight fractional Poincar\'e--Sobolev inequality, i.e. 
\begin{align*}
    \nrm{f-\ip{f}_{Q}}_{L^q_\omega(Q)} &\lesssim (1-s)^{\frac1r} \cdot [f]_{F^{s,\sigma}_{p,r}(Q)}\\&=(1-s)^{\frac1r} \cdot \has{ \int_{Q} \has{\int_{Q} \frac{|f(x)-f(y)|^r}{|x-y|^{d+sr}}\dd y}^{p/r}\sigma(x)^p\dd x}^{1/p}.
\end{align*}
under appropriate conditions on $p,q,r,s,\sigma$ and $\omega$. We note that such inequalities would be significantly simpler to prove without the BBM-factor $(1-s)^{\frac1r}$, see e.g. the proof of Lemma \ref{lem:fracbasic} below in the case $s\leq \frac12$. 
As far as the authors are aware, the results in this section are entirely new when $p\neq r$ or $\omega^q \neq \sigma^p$. We will compare our results to the existing literature for $p= r$ and $\omega^q = \sigma^p$
in Subsection \ref{sec:ABliterature}.

\subsection{Main result}\label{subsec:mainlower} Our main two-weight  fractional Poincar\'e--Sobolev inequality reads as follows.
	\begin{theorem}\label{theorem:IleqII}
		Let $Q\subseteq \R^d$ be a cube, $1\leq p \leq q < \infty$ and $r \in [1,\infty)$, $s \in (0,1)$, $0\leq \alpha < \tfrac1q+\tfrac1{p'}$, $(\omega,\sigma)\in \ApqO$  and assume
		\[
			\varepsilon = \tfrac{s}{d}-\hab{\tfrac1p-\tfrac1q}-\alpha\geq 0.
		\]
        Then we have the following assertions for $f \in F^{s,\sigma}_{p,r}(Q):$
		\begin{enumerate}[(i)]
			\item\label{it:diagpqfraccase1} \emph{(Subcritical case)} If $\varepsilon>0$, we have
			\begin{align*}
				\nrm{f-\ip{f}_{Q}}_{L^q_\omega(Q)} &\lesssim_d  [\omega,\sigma]_{A_{p,q}^\alpha(Q)}\cdot (1-s)^{\frac1r} \cdot    \frac{\abs{Q}^{\varepsilon}}{\varepsilon} [f]_{F^{s,\sigma}_{p,r}(Q)}.
			\end{align*}
			\item\label{it:diagfracpqcase2}  \emph{(Critical case I)} If $\varepsilon=0$ and we additionally assume $p\geq r$ and $\omega^q\in A_\infty(Q)$, then
					\begin{align*}
						\nrm{f-\ip{f}_{Q}}_{L^q_\omega(Q)} &\lesssim_{p,q,r,d} [\omega,\sigma]_{A_{p,q}^\alpha(Q)} \cdot  (1-s)^{\frac1r}\cdot [f]_{F^{s,\sigma}_{p,r}(Q)} \cdot \begin{cases}
				[\omega^q]_{A_\infty(Q)}^{\frac{1}{p'}}\quad & 1<p<q ,\\
				[\omega^q]_{A_\infty(Q)} & \text{otherwise}.
			\end{cases}
					\end{align*}
				\item \label{it:diagfracpqcase3}  \emph{(Critical case II)} If $\varepsilon=0$ and we additionally assume $p>1$ and  $\omega^q,\sigma^{-p'}\in A_\infty(Q)$, then 
				\begin{align*}
					\nrm{f-\ip{f}_{Q}}_{L^q_\omega(Q)} &\lesssim_{p,q,r,d} [\omega,\sigma]_{A_{p,q}^\alpha(Q)} \cdot (1-s)^{\frac1r}\cdot [f]_{F^{s,\sigma}_{p,r}(Q)} \cdot \begin{cases}
				[\omega^q]_{A_\infty(Q)}^{\frac{1}{p'}} + [\sigma^{-p'}]_{A_\infty(Q)}^{\frac{1}{q}} & p<q,\\
				[\omega^q]_{A_\infty(Q)}^{\frac{1}{p'}} \hspace{3pt}\cdot\hspace{3pt} [\sigma^{-p'}]_{A_\infty(Q)}^{\frac{1}{q}} & p=q.
		\end{cases}
				\end{align*}
		\end{enumerate}
	\end{theorem}
	\begin{remark}~\begin{enumerate}[(i)]
        \item The weak implies strong truncation argument in Proposition \ref{prop:weakstrong} requires $p\geq r$; this is the only step where this hypothesis is used and this forces us to treat the critical setting in two different regimes. Moreover, note that neither covers the range $p=1<r$. 
	    \item Critical case I and II in Theorem \ref{theorem:IleqII} are both applicable in  case that $p\geq r$ and $p>1$. Case I is qualitatively stronger in this setting and if $1<p<q$ it is also quantitatively stronger, whereas in the case $p=q$ the results are quantitatively incomparable.
	\end{enumerate}
		
	\end{remark}
	Again, we need some preparation before we can prove the theorem. In the proof of Theorem \ref{theorem:IleqIII}, we used the $(1,1)$-Poincar\'e inequality \eqref{eq:11poincare} as the basis for our proof. In the proof of Theorem \ref{theorem:IleqII} we will replace this inequality by the following lemma.
	\begin{lemma} \label{lem:fracbasic}
		Let $r \in [1,\infty)$, $s \in (0,1)$ and $f \in F^{s}_{1,r}(Q)$, then
		\begin{align*}
			\nrm{f - \ip{f}_Q}_{L^1(Q)} &\lesssim_{r,d}  (1-s)^{\frac1r} \cdot\ell(Q)^s\cdot [f]_{F^{s}_{1,r}(Q)}.
		\end{align*}
	\end{lemma}
\begin{proof}
For $s\geq \frac12$, we know by combining \eqref{eq:avgbelowinf} and \cite[Corollary 3.6]{DLTYY24} with $X = L^1(Q)$ and $k=1$ that
\begin{align*}
     \nrm{f - \ip{f}_Q}_{L^1(Q)} &\lesssim_{r,d}  (1-s)^{\frac1r} \cdot\ell(Q)^s\cdot \has{\int_{\abs{h}\leq \ell(Q)} \frac{\nrmb{\hab{f-f(\cdot +h)}\ind_{Q}(\cdot+h)}^r_{L^1(Q)}}{\abs{h}^{d+sr}}\dd h}^{1/r}\\
    &\leq (1-s)^{\frac1r} \cdot\ell(Q)^s\cdot [f]_{F^{s}_{1,r}(Q)}
\end{align*}
where the second line follows from Minkowski's inequality. 

If $s \leq \frac12$, it suffices to prove the estimate without the BBM-factor $(1-s)^{\frac1r}$. By H\"older's inequality, we have
\begin{align*}
\nrm{f-\ip{f}_Q}_{L^1(Q)}
&\leq
\frac{1}{|Q|}\int_Q\int_Q \abs{f(x)-f(y)}\dd x\dd y \\
&\leq
\frac{1}{|Q|}
\int_Q
\has{
\int_Q \frac{\abs{f(x)-f(y)}^r}{\abs{x-y}^{d+sr}}\dd y
}^{1/r}\cdot  \has{
\int_Q \abs{x-y}^{\frac{(d+sr)r'}{r}}\,\dd y
}^{1/r'} \dd x\\
&\lesssim_d
\frac{\ell(Q)^{\frac{d+sr}{r}} \abs{Q}^{\frac1{r'}}}{|Q|}
\int_Q
\has{
\int_Q \frac{\abs{f(x)-f(y)}^r}{\abs{x-y}^{d+sr}}\dd y
}^{1/r}\dd x\\
&= \ell(Q)^s\cdot [f]_{F^{s}_{1,r}(Q)}
\end{align*}
proving the claim.
\end{proof}
	Let us now prove the main theorem of this section. The outline of the proof is the same as the proof of Theorem \ref{theorem:IleqIII}.
\begin{proof}[Proof of Theorem \ref{theorem:IleqII}]
	For $R \in \mc{D}(Q)$  define
	\[
	f_{R}^{s,r}(x):=\left(\int_R \frac{|f(x)-f(y)|^r}{|x-y|^{d+sr}}\mathrm{d}y\right)^{1/r}, \qquad x \in R.
	\]
	For \ref{it:diagpqfraccase1} we take $\varepsilon>0$. Using Lemma \ref{lem:fracbasic} and H\"older's inequality, we obtain the estimate
	\begin{align}
    \begin{split}
		\frac{1}{\abs{R}} \int_R \abs{f-\ip{f}_R} \dx &\lesssim_{r,d} (1-s)^{\frac1r} \frac{\ell(R)^s}{\abs{R}} \int_R f_{R}^{s,r}\dx \\
		&\leq  (1-s)^{\frac1r}\frac{1}{\abs{R}^{1-\frac{s}{d}}} \has{\int_R (f_{R}^{s,r})^p \sigma^p\dx}^{1/p} \has{\int_R \sigma^{-p'}\dx}^{1/p'}\\
		&\leq (1-s)^{\frac1r} [\omega,\sigma]_{A_{p,q}^\alpha({Q})} \frac{\abs{R}^{\varepsilon}}{\nrm{\omega}_{L^q(R)}} \nrm{f^{s,r}_{Q}}_{L^p_\sigma(R)}.\end{split}\label{eq:avgcontrolAB}
	\end{align}
	The result now follows by using \eqref{eq:basicdom} and  Lemma \ref{lem:QtoQ0}.
    
	For \ref{it:diagfracpqcase2} and \ref{it:diagfracpqcase3}, assume $\omega^q\in A_\infty(Q)$, take $\varepsilon=0$ and let $\beta = \alpha+\frac1p - \frac 1q = \frac{s}{d}\in (0,1)$.
	Using Lemma \ref{lem:fracbasic}, we have
	\[
		\ip{\abs{f-\ip{f}_R}}_R \lesssim_{r,d} |R|^{\beta}(1-s)^{1/r} \ipb{{f_{R}^{s,r}}}_{R}.
	\]
	Moreover, by Lemma \ref{lem:sparsedom}, there exists a sparse collection of cubes $\mc{S}\subseteq \mc{D}({Q})$ such that
	\begin{align*}
		\abs{f(x)-\ip{f}_{{Q}}} &\lesssim_d \sum_{R \in \mc{S}} \ip{\abs{f - \ip{f}_R}}_{R} \ind_R(x) \\
		&\lesssim_{r,d} (1-s)^{\frac1r}\sum_{R \in \mc{S}}|R|^{\beta} \ipb{{f_{R}^{s,r}}}_{R}\mathbf{1}_R(x) \\
		&\leq (1-s)^{\frac1r} \cdot \mathcal{A}^{1,\beta}_\mathcal{S}(f_{Q}^{s,r}).
	\end{align*}
	If $p\geq r$, we use Proposition \ref{prop:weaksparse}\ref{it:weak} to get
	\begin{align*}
		\nrmb{f-\ip{f}_{Q} }_{L^{q,\infty}_\omega({Q})} \lesssim_{p,q,d}  [\omega,\sigma]_{A_{p,q}^\alpha({Q})} \cdot (1-s)^{\frac1r} \cdot \nrmb{f_{Q}^{s,r}}_{L^p_\sigma({Q})} \cdot \begin{cases}
			[\omega^q]_{A_\infty(Q)}^{\frac{1}{p'}}, &1<p<q,\\
			[\omega^q]_{A_\infty(Q)}, \quad&\text{otherwise} .
		\end{cases}
	\end{align*}
	Hence, critical case I follows from the weak implies strong principle, see Proposition \ref{prop:weakstrong}, by using \eqref{eq:avgcontrolAB} with $R=Q$ and $\varepsilon=0$ and noting that $[\omega^q]_{A_\infty(Q)}\geq 1$. If $p>1$ and $\sigma^{-p'}\in A_\infty(Q)$, we use Proposition \ref{prop:weaksparse}\ref{it:strong}, proving critical case II.
\end{proof}

As in Corollary \ref{cor:AleqConeweight}, we can specialize Theorem \ref{theorem:IleqII} to the one-weight case using Lemma \ref{lem:ApApq}. The proof follows the same lines as the proof of Corollary \ref{cor:AleqConeweight}. As in that corollary, we note that the four cases in the constant below are not mutually exclusive. 

        \begin{corollary}\label{cor:AleqBoneweight}
	Let $Q\subseteq \R^d$ be a cube, $1\leq u\leq p \leq q < \infty$ and $r \in [1,\infty)$, $s \in (0,1)$, $w\in A_u$, and assume
		\[
			\varepsilon = \tfrac{s}{du}-\hab{\tfrac1p-\tfrac1q}\geq 0.
		\]
Then we have for $f \in F^{s,\sigma}_{p,r}(Q)$  with $\sigma = w^{1/p}$
		\begin{align*}
			\has{\frac{1}{w(Q)}\int_{Q}\abs{f-\ip{f}_{Q}}^qw\dd x}^\frac1q \lesssim \ell(Q)^s &\cdot  (1-s)^{\frac1r} \cdot \has{ \frac{1}{w(Q)}\int_{Q} \has{\int_{Q} \frac{|f(x)-f(y)|^r}{|x-y|^{d+sr}}\dd y}^{\frac{p}{r}}w(x)\dd x}^{\frac{1}{p}} \\&\cdot [w]_{A_p}^{\frac{1}{p}} [w]_{A_u}^{\frac{1}{p}-\frac1q} \cdot \begin{cases}
			    \tfrac{1}{\varepsilon}, &\varepsilon>0,\\
                [w]_{A_\infty}^{\frac{1}{p'}},  & 1< p <q, \, p\geq r\\
                [w]_{A_\infty},  & 1=r= p,\\
                {[w]_{A_\infty}^{\frac{1}{p'}}+\bracb{w^{-\frac{1}{p-1}}}_{A_\infty}^{\frac{1}{q}}},& 1<p<q.
			\end{cases}
		\end{align*}
        where the implicit constant only depends on $p,q,r$ and $d$.
    \end{corollary}

\begin{proof}
    This follows from Theorem \ref{theorem:IleqII}  exactly as Corollary \ref{cor:AleqConeweight} followed from Theorem \ref{theorem:IleqIII}.
\end{proof}
    
    \subsection{Comparison to the literature}\label{sec:ABliterature}
In this subsection, we compare Theorem \ref{theorem:IleqII} and Corollary \ref{cor:AleqBoneweight} with earlier results in the literature. Our two-weight fractional Poincar\'e--Sobolev inequalities  with $p\neq r$ or  $\omega^q \neq \sigma^p$ with the sharp BBM-factor $(1-s)^{\frac1r}$ seem to be entirely new, so we focus on the one-weight case, i.e.  Corollary \ref{cor:AleqBoneweight}, for $p=r$. 
\begin{enumerate}[(i)]
    \item Qualitatively, the subcritical case in Corollary \ref{cor:AleqBoneweight} for a specific $\varepsilon>0$, $p=r$ and $u=1$ was obtained in \cite{HMPV23b}.
Indeed, for $q\geq p$ satisfying
    \begin{align*}
        \frac1p-\frac1q = \frac {s}{d}\cdot\frac{1}{1 +\log[w]_{A_1}} \qquad &\Longrightarrow \qquad \varepsilon = \frac{s}d\cdot\frac{ \log[w]_{A_u}}{u(u+\log[w]_{A_u})},
    \end{align*}
    \cite[Theorem 2.1]{HMPV23b} yields the estimate Corollary \ref{cor:AleqBoneweight} with weight dependence 
    $
[w]_{A_1}^{1/p+1}.
    $
    This was improved to general $1 \leq u \leq p=r$ in \cite[Theorem 5.9]{MPW24} with weight dependence
    $$
[w]_{A_p}^{\frac{1}p}[w]_{A_\infty}.
    $$
    at the cost of a singularity for $s \to 0$ of the form $s^{-\frac1{p'}}$. In Corollary \ref{cor:AleqBoneweight}, besides allowing $p\neq r$, we obtain for this specific $\varepsilon>0$ and $p>1$
    the weight dependence
    $$
    [w]_{A_p}^{\frac{1}p}\min\cbraceb{\tfrac{1}{s},[w]_{A_\infty}^{\frac{1}{p'}}},
    $$
     offering the choice between either a singularity as $s\to 0$ while completely removing the $[w]_{A_\infty}$ factor or no singularity as $s\to 0$ and an improvement of the power on $[w]_{A_\infty}$ from $1$ to $\frac1{p'}$. The case $p=1$ also holds in our case, with power $1$ on $[w]_{A_\infty}$ as in \cite[Theorem 5.9]{MPW24}.
    \item The critical case $\varepsilon=0$ in Corollary \ref{cor:AleqBoneweight} for $p=r$ and $u=1$ was obtained in \cite[Theorem 2.3]{HMPV23b} with weight dependence
    $$
[w]_{A_1}^{\frac{2}p+1-\frac{1}{q}}.
    $$
    This was improved to general $1 \leq u \leq p=r$ in \cite[Theorem 5.7]{MPW24} with weight dependence
    $$
[w]_{A_p}^{\frac{1}p}[w]_{A_u}^{\frac{1}p-\frac1q}[w]_{A_\infty}.
    $$
    at the cost of a singularity for $s \to 0$ of the form $s^{-\frac1{p'}}$. Besides allowing $p\neq r$, we remove the singularity and improve the power on $[w]_{A_\infty}$ from $1$ to $\frac1{p'}$ for $p>1$ in Corollary \ref{cor:AleqBoneweight}.
    \item The case $p=r$, $u=1$ and $\varepsilon = \frac{1}{p'}$ of Corollary \ref{cor:AleqBoneweight}  was very recently considered in \cite[Theorem 2.7]{Cla25}. The quantitative dependence on the weight characteristics in \cite{Cla25} in the case $p=1$ is sharper than Corollary \ref{cor:AleqBoneweight}, at the expense of an added singularity for $s \to 0$ of the form $s^{-1-\frac{1}{p'}}$. For $p>1$ the results are quantitatively incomparable.
\end{enumerate}
The above mentioned results in \cite{Cla25} are proven via the following local fractional subrepresentation formula in  \cite[(5.1)]{Cla25}
\begin{equation*}
\abs{f(x)-\ip{f}_{Q}} \lesssim_d \ell(Q)^{\frac{s}{r'}} \cdot \frac{(1-s)^{\frac{1}{r}}}{s^{\frac{1}{r '}}}  \cdot \has{\int_{Q} \frac{\abs{f^{s,r}_Q(y)}^r}{|x-y|^{d-s}}}^\frac{1}{r} \dd y, \qquad x \in Q,
    \end{equation*}
    where
    $$
f^{s,r}_Q(y):= \has{\int_Q\frac{\abs{f(y)-f(z)}^r}{\abs{y-z}^{d+sr}}}^{\frac1r}, \qquad y \in Q.
    $$
This formula allows an alternative proof strategy to Theorem \ref{theorem:IleqII} via the boundedness of the fractional integral operator $I_s$ from $L^{p/r}_{\sigma^r}(\R^d)$ to $L^{q/r,\infty}_{\omega^r}(\R^d)$,  which can be found in \cite[Theorem 2.2]{CM13}. However, such an approach  will yield the condition $(\omega^r,\sigma^r)\in A_{p/r,q/r}^{r\alpha}(Q)$, which is more restrictive than $(\omega,\sigma)\in A_{p,q}^\alpha(Q)$ unless $r=1$.

\section{Two-weight Sobolev to Triebel–Lizorkin embedding}\label{sec:BC}
To complete our Poincar\'e--Sobolev sandwich, we finally turn to the two-weight Sobolev to Triebel–Lizorkin embedding, i.e. 
\begin{align*}
     (1-s)^{\frac1r} \cdot [f]_{F^{s,\omega}_{q,r}(Q)}\lesssim  \nrmb{|\nabla f|}_{L^{p}_{\sigma}(Q)}
\end{align*}
under appropriate conditions on $p,q,r,s,\sigma$ and $\omega$.

\subsection{Main result}
In our main theorem in this setting we will introduce an additional parameter $p_0\in [1,p]$, which allows more flexibility in the range of $p$. The price for this extra flexibility is that, in the $A_{p,q}^\alpha(Q)$-Muckenhoupt characteristic, one has to replace $p$ by a smaller exponent $p_1\in[1,p]$. If $p_0=1$, then this loss is absent and one simply has $p_1=p$.
\begin{theorem} \label{theorem:IIleqIII}
    Let $Q\subseteq \R^d$ be a cube, $1\leq p_0\leq  p\leq q<\infty$ and $r\in [1,\infty)$ such that $\frac{1}{p_0} -\frac1r \leq \frac1d$ and let $$
\tfrac{1}{p_1}:= 1+\tfrac{1}{p}-\tfrac{1}{p_0}.
    $$
    Let $s \in (0,1)$, $0\leq \alpha <\tfrac1q+\tfrac1{p_1'}$, $(\omega,\sigma)\in A_{p_1,q}^\alpha(Q)$ and assume
    \begin{align*}
        \varepsilon = \tfrac{1-s}{d} -\hab{\tfrac1p - \tfrac1q}-\alpha \geq 0.
    \end{align*}
    Then we have the following assertions for $f\in W^{1,p}_\sigma(Q)$:
	\begin{enumerate}[(i)]
		\item \label{it:BC1}\emph{(Subcritical case)} If $\varepsilon>0$, we have 
        \begin{align*}
			[f]_{F^{s,\omega}_{q,r}(Q)}	\lesssim_{d}[\omega,\sigma]_{A^\alpha_{p_1,q}(Q)}\cdot\frac{|Q|^{\varepsilon}}{\varepsilon^{1/\gamma}} \nrmb{|\nabla f|}_{L^p_\sigma(Q)}
                    \cdot(\varepsilon +\tfrac{s}{d})^{-1}.
		\end{align*}
        where $\gamma=\min\{q,r\}$. If  we additionally assume $ p> p_0$ and  $\sigma^{-{p'_1}}\in A_\infty(Q)$, then
        \begin{align*}
			[f]_{F^{s,\omega}_{q,r}(Q)}	\lesssim_{p,q,\alpha,d}[\omega,\sigma]_{A^\alpha_{p_1,q}(Q)}\cdot \frac{|Q|^{\varepsilon}}{\varepsilon^{1/r}}\nrmb{|\nabla f|}_{L^p_\sigma(Q)} \cdot [\sigma^{-p_1'}]_{A_\infty(Q)}^{\frac1q}.
		\end{align*}
		\item  \label{it:BC2}\emph{(Critical case)} If $\varepsilon=0$ and we additionally assume $1<p_0<p$, $\tfrac{1}{p_0}-\tfrac1r < \tfrac1d$ and $\omega^q,\sigma^{-p_1'}\in A_\infty(Q)$, then
        \begin{align*}
            [f]_{F^{s,\omega}_{q,r}(Q)}\lesssim_{p,p_0,q,\alpha,d} &\frac{[\omega,\sigma]_{A_{p_1,q}^\alpha(Q)}}{s^{1+1/r}\cdot (1-s)^{1/r}} \cdot\nrmb{|\nabla f|}_{L^{p}_{\sigma}(Q)}   \cdot  \begin{cases}
				[\omega^{q}]_{A_\infty(Q)}^{\frac1{p'}} + [\sigma^{-p_1'}]_{A_\infty(Q)}^{\frac{1}{q}}\quad & p<q,\vspace{3pt}\\
				[\omega^q]_{A_\infty(Q)}^{\frac{1}{p'}} \hspace{3pt}\cdot\hspace{3pt} [\sigma^{-p_1'}]_{A_\infty(Q)}^{\frac{1}{q}} & p=q.
			\end{cases}
        \end{align*}
	\end{enumerate}
\end{theorem}

Before turning to the proof of Theorem \ref{theorem:IIleqIII}, let us make a few remarks on the statement.
\begin{remark}~
\begin{enumerate}[(i)] 
    \item For the first inequality of the subcritical case of Theorem \ref{theorem:IIleqIII}, we can relax the assumption on $\alpha$ to $\alpha < \tfrac1q + \tfrac1{p'}$. In particular, this includes weights $(\omega,\sigma)$ which are in $\Apq$ for $\alpha = \tfrac1q + \tfrac1{p_1'}$, but not for any $\alpha <\tfrac1q + \tfrac1{p_1'}$.
    \item If, in the subcritical case of Theorem \ref{theorem:IIleqIII}, we additionally assume $\omega^q\in A_\infty(Q)$, we can use Theorem \ref{theorem:IleqIII}\ref{it:diagpqcase2} combined with Lemma \ref{lem:propertiesMuckenhoupt}\ref{it:weightalphabeta} in the proof instead of Theorem \ref{theorem:IleqIII}\ref{it:diagpqcase1}, allowing us to replace   $(\varepsilon+\tfrac{s}{d})^{-1}$ with an appropriate power of $[\omega^q]_{A_\infty(Q)}$, i.e.
        \begin{align*}
			[f]_{F^{s,\omega}_{q,r}(Q)}	\lesssim_{d}[\omega,\sigma]_{A^\alpha_{p_1,q}(Q)}\cdot\frac{|Q|^{\varepsilon}}{\varepsilon^{1/\gamma}} \nrmb{|\nabla f|}_{L^p_\sigma(Q)}
                    \cdot \begin{cases}
				[\omega^{q}]_{A_\infty(Q)}^{1/p'} \quad & 1<p<q,\\
				[\omega^q]_{A_\infty(Q)}  & \text{otherwise}.
			\end{cases}
		\end{align*}   
    \item Note that we always have $d\varepsilon \leq 1-s$, with equality if and only if $p=q$ and $\alpha=0$. Thus, in the subcritical case, $\varepsilon$ plays the role of the quantity $(1-s)$, although in general it may be smaller. The exponent on $\varepsilon$ in the subcritical case is of the expected order on $(1-s)$.
 \item As we shall show in Proposition \ref{prop:counterexample_jumpq}, the factor
    $\varepsilon^{-1/q}$ in the estimate of
    Theorem \ref{theorem:IIleqIII}\ref{it:BC1} is sharp in the cases
    $1=d=p\leq q<r$ or $1=p=q$. For the case $1=p<q<r$ and $d\geq 2$ in Theorem \ref{theorem:IIleqIII}\ref{it:BC1}, we do not know the sharp factor.
    \item If in the critical case of Theorem \ref{theorem:IIleqIII} we either have $p_0=1$ or $\tfrac1{p_0}-\tfrac1r=\tfrac1d$, the inequality still holds if we replace $(1-s)^{1/r}$ by $(1-s)^{1/p_0}$. However, since $p_0$ is an auxiliary parameter where an infinitesimal small change often does not matter, we have excluded this case from the statement of the theorem. The proof changes only in the place where one estimates $T_2$ using the subcritical case of this theorem, where now one cannot find a $p_0^*$ that satisfies the given conditions. Therefore, using $p_0$ instead of $p_0^*$ leads to the exponent $\frac1{p_0}$ on $1-s$.
    \end{enumerate}
\end{remark}

\begin{proof}[Proof of Theorem \ref{theorem:IIleqIII}]
    Let us first take $\varepsilon>0$. By Lemma \ref{lemma:Tldomsubcrit}, for  a.e. $x \in \R^d$ we have the pointwise estimate
    \begin{align*}
        \has{\int_Q \frac{|f(x)-f(y)|^r}{|x-y|^{d+sr}}\dd y}^{1/r} &\lesssim_d \has{\sum_{R\in \mathcal{D}({Q})} \ell(R)^{-sr}{|f(x)-\ip{f}_{R}|^r}\ind_R(x)}^{1/r},\\
      &\hspace{1cm}+ \has{ \sum_{R\in \mathcal{D}({Q})}\ell(R)^{-sr}   \ipb{|f-\ip{f}_{3R}|}_{r,3R}^r \ind_R(x)}^{1/r}\\
    &=: g(x)+h(x).
    \end{align*}
    We will estimate $\nrm{g}_{L^q_\omega(Q)}$ and $\nrm{h}_{L^q_\omega(Q)}$ to prove the subcritical case. We first prove the case without additional assumptions. Define $\gamma=\min\{q,r\}$. Using the embedding $\ell^\gamma\hookrightarrow \ell^r$, we have
    \begin{align*}
        g(x) \leq \has{\sum_{R\in \mathcal{D}({Q})} \ell(R)^{-s\gamma}{|f(x)-\ip{f}_{R}|^\gamma}\ind_R(x)}^{1/\gamma}.
    \end{align*}
    Using Minkowski's inequality twice, we have 
	\begin{align*}
		\nrm{g}_{L^q_\omega(Q)}&\leq \has{\sum_{j=0}^\infty \has{ \int_{Q}\has{ \sum_{\substack{R\in \mathcal{D}({Q}):\\ \ell(R)=2^{-j}\ell({Q})}}\ell(R)^{-s}{|f(x)-\ip{f}_{R}|}\ind_R(x)}^q\omega(x)^q\dx }^{\gamma/q}}^{1/\gamma}\\
        &= \ell(Q)^{-s}\has{\sum_{j=0}^\infty 2^{js\gamma}\has{ \int_{Q}\has{ \sum_{\substack{R\in \mathcal{D}({Q}):\\ \ell(R)=2^{-j}\ell({Q})}}{|f(x)-\ip{f}_{R}|^p}\ind_R(x)}^{q/p}\omega(x)^q\dx }^{\gamma/q}}^{1/\gamma}\\
        &\leq \ell(Q)^{-s}\has{\sum_{j=0}^\infty 2^{js\gamma}\has{\sum_{\substack{R\in \mathcal{D}({Q}):\\ \ell(R)=2^{-j}\ell({Q})}}\nrm{f-\ip{f}_{R}}_{L^q_\omega(R)}^p} ^{\gamma/p}}^{1/\gamma},
    \end{align*}
    where we used that $p\leq q$ and that cubes $R\in \mc{D}(Q)$ with the same side length are pairwise disjoint.
    Now using the two-weight $(p,q)$-Poincar\'e inequality in Theorem \ref{theorem:IleqIII}\ref{it:diagpqcase1}, we get
    \begin{align*}
        \nrm{g}_{L^q_\omega(Q)}&\lesssim_d [\omega,\sigma]_{\Apq}\cdot (\varepsilon+\tfrac{s}{d})^{-1}\cdot |Q|^{\varepsilon}\has{\sum_{j=0}^\infty 2^{-j\gamma d\varepsilon}\has{\sum_{\substack{R\in \mathcal{D}({Q}):\\ \ell(R)=2^{-j}\ell({Q})}}\nrmb{|\nabla f|}_{L^p_\sigma(R)}^p}^{\gamma/p}}^{1/\gamma} \\
        &=[\omega,\sigma]_{\Apq}\cdot (\varepsilon+\tfrac{s}{d})^{-1}\cdot|Q|^{\varepsilon}\nrmb{|\nabla f|}_{L^p_\sigma(Q)}\has{\sum_{j=0}^\infty 2^{-j\gamma d\varepsilon}}^{1/\gamma} \\
        &\lesssim_d  [\omega,\sigma]_{A^\alpha_{p_1,q}(Q)}\cdot (\varepsilon+\tfrac{s}{d})^{-1}\cdot\frac{|Q|^{\varepsilon}}{\varepsilon^{1/\gamma}} \nrmb{|\nabla f|}_{L^p_\sigma(Q)},
    \end{align*}
where we used $p_1\leq p$ and Lemma \ref{lem:propertiesMuckenhoupt}\ref{it:weightpq}.

    For $h$, we also use the embedding $\ell^\gamma \hookrightarrow \ell^r$, Minkowski and then the unweighted $(p_0,r)$-Poincar\'e inequality to estimate
    \begin{align*}
        \nrm{h}_{L^q_\omega(Q)}&\leq \has{\sum_{j=0}^\infty \has{\int_Q\has{\sum_{\substack{R\in \mathcal{D}({Q}):\\ \ell(R)=2^{-j}\ell(Q)}}\ell(R)^{-s}    \ipb{|f-\ip{f}_{3R}|}_{r,3R} }^{q}\omega(x)^q\dx }^{\gamma/q}}^{1/\gamma} \\
        &\lesssim_{d} \has{\sum_{j=0}^\infty \has{\int_Q\has{\sum_{\substack{R\in \mathcal{D}({Q}):\\ \ell(R)=2^{-j}\ell(Q)}}\ell(R)^{1-s-d/p_0}    \nrmb{|\nabla f|}_{L^{p_0}(3R)} }^{q}\omega(x)^q\dx }^{\gamma/q}}^{1/\gamma} \\
        &\leq \has{\sum_{j=0}^\infty \has{\int_Q\has{\sum_{\substack{R\in \mathcal{D}({Q}):\\ \ell(R)=2^{-j}\ell(Q)}}\ell(R)^{1-s-d/p_0}    \nrmb{|\nabla f|}_{L^p_\sigma(3R)}\cdot \sigma^{-p_1'}(3R)^{1/p_1'}}^{q}\omega(x)^q\dx }^{\gamma/q}}^{1/\gamma},
    \end{align*} 
    where we used H\"older in the last step. Since
    \[
    \sigma^{-p_1'}(3R)^{1/p_1'}\lesssim_d \frac{[\omega,\sigma]_{A^{\alpha}_{p_1,q}(Q)}}{\omega^q(R)^{1/q}}|R|^{\frac1q + \frac{1}{p_1'}-\alpha},
    \]
    we have
    \begin{align*}
        \nrm{h}_{L^q_\omega(Q)} \lesssim_d [\omega,\sigma]_{A^{\alpha}_{p_1,q}(Q)}\has{\sum_{j=0}^\infty \has{\int_Q\has{\sum_{\substack{R\in \mathcal{D}({Q}):\\ \ell(R)=2^{-j}\ell(Q)}}\frac{|R|^{\varepsilon}}{\omega^q(R)^{1/q}}    \nrmb{|\nabla f|}_{L^p_\sigma(3R)} }^{q}\omega(x)^q\dx }^{\gamma/q}}^{1/\gamma},
    \end{align*}
    which can be estimated exactly as in the proof of Lemma \ref{lem:QtoQ0} (which treats the case $\gamma=1$)    to obtain
    \begin{align*}
        \nrm{h}_{L^q_\omega(Q)} \lesssim_d [\omega,\sigma]_{A^{\alpha}_{p_1,q}(Q)} \cdot \frac{|Q|^{\varepsilon}}{\varepsilon^{1/\gamma}} \nrmb{|\nabla f|}_{L^p_\sigma(Q)}.
    \end{align*}

    Now suppose $p>p_0$ and $\sigma^{-{p'_1}}\in A_\infty(Q)$. Fix $x\in Q$ and let $\cbrace{Q_j}_{j=1}^\infty$ be the sequence of dyadic cubes in $\mc{D}(Q)$ such that $\ell(Q_j) = 2^{-j}\ell(Q)$ and $x \in Q_j$ for all $j\geq 0$. To estimate $g$, we note that by combining \eqref{eq:basicdom} and \eqref{eq:11poincare}, we have for all $j\geq 0$
\begin{align*}
    |f(x)-\ip{f}_{Q_j}| \lesssim_d \sum_{k=j}^\infty \ipb{|f-\ip{f}_{Q_k}|}_{Q_k} &\lesssim_d   \sum_{k=j}^\infty \ell(Q_k) \ipb{\abs{\nabla f}}_{Q_k}\\
    &\leq M_{Q}^{\beta}\hab{\abs{\nabla f}}(x) \cdot \sum_{k=j}^\infty 2^{j-k} \ell(Q_j)\cdot \abs{Q_j}^{-\beta}\\
    &= 2\, \ell(Q_j)^{1-d(\frac1p-\frac1q)-d\alpha} \cdot M_{Q}^{\beta}\hab{\abs{\nabla f}}(x),
\end{align*}
where $\beta = \frac1p-\frac1q+\alpha\in (0,1/p_0)$.
Therefore,
\begin{align*}
    g(x) &\lesssim_d 
    \has{\sum_{j=0}^\infty \ell(Q_j)^{\varepsilon dr} }^{1/r}\cdot M_{Q}^{\beta}\hab{\abs{\nabla f}}(x) \lesssim \frac{|Q|^{\varepsilon}}{\varepsilon^{1/r}} \cdot M_{Q}^{\beta}\hab{\abs{\nabla f}}(x).
\end{align*}
By boundedness of $M_{Q}^{\beta}:L_\sigma^p(Q)\to L_\omega^q(Q)$ (see Corollary \ref{cor:strongmax}), we conclude
\begin{align*}
    \nrm{g}_{L^q_\omega(Q)} &\lesssim_{d} [\omega,\sigma]_{\Apq} \cdot \frac{|Q|^{\varepsilon}}{\varepsilon^{1/r}} \nrmb{\abs{\nabla f}}_{L^p_\sigma(Q)} \cdot [\sigma^{-p'}]_{A_\infty(Q)}^{\frac1q}.
\end{align*}
Now note that $[\omega,\sigma]_{\Apq} \leq [\omega,\sigma]_{A_{p_1,q}^\alpha(Q)}$ by Lemma \ref{lem:propertiesMuckenhoupt}\ref{it:weightpq} and $[\sigma^{-p'}]_{A_\infty(Q)} \leq [\sigma^{-p_1'}]_{A_\infty(Q)}$ by \eqref{eq:Ainftyembeds}.

To estimate $h$, we again use the unweighted $(p_0,r)$-Poincar\'e inequality to obtain
\begin{align*}
  \has{\int_{3 {Q_j}} {|f(y)-\ip{f}_{3Q_j}|^r}\dd y}^\frac1r    &\lesssim_d |Q_j|^{\frac1d-\frac1{p_0}+\frac1r}\has{\int_{3 {Q_j}} \abs{\nabla{f}}^{p_0}\dd y}^{1/{p_0}}\\
     &\lesssim_d \ell(Q_j)^{1+\frac{d}r-d(\frac1p-\frac1q)-d\alpha}\cdot \has{M^{\beta p_0}_Q\hab{\abs{\nabla f}^{p_0}}}^\frac1{p_0}.
\end{align*}
Therefore we can estimate
\begin{align*}
    h(x) \lesssim_d 
    \has{\sum_{j=0}^\infty \ell(Q_j)^{\varepsilon dr} }^{1/r}\cdot \has{M^{\beta p_0}_Q\hab{\abs{\nabla f}^{p_0}}(x)}^\frac1{p_0} \lesssim \frac{|Q|^{\varepsilon}}{\varepsilon^{1/r}} \cdot \has{M^{\beta p_0}_Q\hab{\abs{\nabla f}^{p_0}}(x)}^\frac1{p_0}.
\end{align*}
Using the boundedness of $M^{\beta p_0}_Q:L_{\sigma^{p_0}}^{p/p_0}(Q)\to L_{\omega^{p_0}}^{q/p_0}(Q)$ (see Corollary \ref{cor:strongmax} with $(\omega,\sigma,p,q,\alpha)$ replaced by $(\omega^{p_0},\sigma^{p_0},\frac{p}{p_0},\frac{q}{p_0},p_0\alpha)$), we obtain
\begin{align*}
    \nrm{h}_{L^q_\omega (Q)} \lesssim_{p,q,d} [\omega,\sigma]_{A_{p_1,q}^\alpha(Q)}\cdot \frac{|Q|^{\varepsilon}}{\varepsilon^{1/r}} \nrmb{\abs{\nabla f}}_{L^p_\sigma(Q)}\cdot [\sigma^{-p_1'}]_{A_\infty(Q)}^{\frac1q},
\end{align*}
where we also used Lemma \ref{lem:propertiesMuckenhoupt}\ref{it:alphat} and $p_0(p/p_0)'=p_1'$, finishing the proof of subcritical case.

\medskip

  In the critical case, i.e. if $\varepsilon=0$, we use Theorem \ref{thm:fracsparsedom} to find a sparse collection of cubes $\mc{S}\subseteq \mc{D}(Q)$ such that
	\begin{align*}
		[f]_{F^{s,\omega}_{q,r}(Q)} &\lesssim_d \bigg( \int_{Q} \Big(\sum_{R\in \mathcal{S}} \frac{\ell(R)^{-s}}{s^{1+1/r}}\ipb{\abs{f-\ipb{f}_R}}_R\mathbf{1}_R\Big)^q\omega^q \dx\bigg)^{1/q} \\
		&\qquad + \bigg( \int_{Q} \Big(\sum_{R\in \mathcal{S}} \ipb{f^{s,r}_{3R}}_R\mathbf{1}_R\Big)^q\omega^q \dx\bigg)^{1/q}\\
        &=:T_1+T_2.
	\end{align*}
    As before, define
    $\beta=\tfrac1p -\tfrac1q +\alpha = \tfrac{1-s}{d}.$
    Using the $(1,1)$-Poincar\'e inequality \eqref{eq:11poincare}, we get
    \begin{align*}
        \sum_{R\in \mc{S}} \ell(R)^{-s}\ipb{|f-\ip{f}_R|}_R \mathbf{1}_R &\lesssim \sum_{R\in \mc{S}}  |R|^{\beta} \ipb{|\nabla f|}_{1,R} \mathbf{1}_R = \mc{A}^{1,\beta}_\mc{S}(|\nabla f|).
    \end{align*}
    Therefore, we can use Proposition \ref{prop:weaksparse}\ref{it:strong}  to obtain 
	\begin{align*}
		T_1 \lesssim_{p,q,d} \frac{[\omega,\sigma]_{\Apq}}{s^{1+1/r}} \nrmb{|\nabla f|}_{L^p_\sigma(Q)}\cdot \begin{cases}
				[\omega^q]_{A_\infty(Q)}^{\frac{1}{p'}} + [\sigma^{-p'}]_{A_\infty(Q)}^{\frac{1}{q}}\quad & p<q,\\
				[\omega^q]_{A_\infty(Q)}^{\frac{1}{p'}} \hspace{3pt}\cdot\hspace{3pt} [\sigma^{-p'}]_{A_\infty(Q)}^{\frac{1}{q}} & p=q.
			\end{cases}
	\end{align*}
    Now use again that $[\sigma^{-p'}]_{A_\infty(Q)} \leq [\sigma^{-p_1'}]_{A_\infty(Q)}$ and $[\omega,\sigma]_{\Apq} \leq [\omega,\sigma]_{A_{p_1,q}^\alpha(Q)}$.
    
    For $T_2$, let $p_0^*$ be such that $1<p_0^*<p_0$ and $\tfrac{1}{p_0^*}-\tfrac1r < \tfrac1d$. Then we first use H\"older and then the above proven subcritical case of this theorem with $\omega=\sigma=1$, $\alpha=0$ and $(p_0,p,q)=(p_0^*,p_0,p_0)$ to obtain
	\begin{align*}
		\ipb{f^{s,r}_{3R}}_R &\leq |R|^{-1/p_0}\nrm{f^{s,r}_{3R}}_{L^{p_0}(R)}\leq |R|^{-1/p_0} [f]_{F^{s}_{p_0,r}(3R)} \lesssim_{d} \frac{\abs{R}^{\frac{1-s}{d}-\frac{1}{p_0}}}{(1-s)^{1/r}} \nrmb{|\nabla f|}_{L^{p_0}(3R)}.
	\end{align*}
    Since
    \begin{align*}
        \frac{\abs{R}^{\frac{1-s}{d}-\frac{1}{p_0}}}{(1-s)^{1/r}} \nrm{|\nabla f|}_{L^{p_0}(3R)} \lesssim_d \frac{\abs{R}^{\frac{1-s}{d}}}{(1-s)^{1/r}} \ipb{|\nabla f|^{p_0}}_{1,3R}^{1/p_0},
    \end{align*}
    and $\frac{1-s}{d} = \frac1p-\frac1q + \alpha$, using Lemma \ref{lem:propertiesMuckenhoupt}\ref{it:alphat} we can again apply Proposition \ref{prop:weaksparse}\ref{it:strong} with $(\omega,\sigma,p,q,r,\alpha)$ replaced by $(\omega^{p_0},\sigma^{p_0},\frac{p}{p_0},\frac{q}{p_0},\frac1{p_0},p_0\alpha)$  to obtain 
    \begin{align*}
        T_2 &\lesssim_d \frac{1}{(1-s)^{1/r}}  \nrmb{\mc{A}_\mc{S}^{1/p_0,p_0\beta}(|\nabla f|^{p_0})}_{L^{q/p_0}_{\omega^{p_0}}(Q)}^{1/p_0} \\
        &\lesssim_{p,q,d} \frac{[\omega,\sigma]_{A_{p_1,q}^\alpha(Q)}}{(1-s)^{1/r}} \nrmb{|\nabla f|}_{L^{p}_{\sigma}(Q)}     \begin{cases}
				[\omega^{q}]_{A_\infty(Q)}^{1/p'} + [\sigma^{-p_1'}]_{A_\infty(Q)}^{\frac{1}{q}}\quad & p<q,\\
				[\omega^q]_{A_\infty(Q)}^{1/p'} \hspace{3pt}\cdot\hspace{3pt} [\sigma^{-p_1'}]_{A_\infty(Q)}^{\frac{1}{q}} & p=q.
			\end{cases}
    \end{align*}
This finishes the proof.
\end{proof}

The following proposition shows that in the unweighted cases $1=p= q$ and $1=d=p\leq q$ of Theorem \ref{theorem:IIleqIII}, the factor $\varepsilon^{-1/q}$ is sharp.
\begin{proposition}\label{prop:counterexample_jumpq}
Let $Q$ be a cube in $\R^d$, $q,r\in[1,\infty)$ and $s\in (0,1)$ such that $\tfrac12 < sq < 1$. Then there exists an $f \in C^\infty({Q})$ such that
\[
\|\nabla f\|_{L^1(Q)}\eqsim_d 1
\qquad\text{and}\qquad
[f]_{F^{s}_{q,r}(Q)}\gtrsim_{d}(1-sq)^{-1/q}.
\]
In particular, if there is a constant $C>0$, independent of $s$, such that for all $f\in C^\infty(Q)$
\[
[f]_{F^{s}_{q,r}(Q)}\le C\,(1-sq)^{-\alpha}\,\|\nabla f\|_{L^1(Q)},
\]
then $\alpha \geq \frac1q$.
\end{proposition}
\begin{proof}
Without loss of generality, we may take $Q = [-1,1]^d$. Furthermore, we may reduce to $d=1$ by tensoring with a constant function in the remaining $(d-1)$-variables. Indeed, for $f(x)=g(x_1)$ one can integrate out the other variables. Indeed, we have
\[
    \|\nabla f\|_{L^1(Q)} \eqsim_d \int_{-1}^1 |g'(x_1)|\dd x_1
\]
and
\[
[f]_{F^{s}_{q,r}(Q)}^q \gtrsim_d \int_{-1}^1\Big(\int_{-1}^1 \frac{|g(x_1)-g(y_1)|^r}{|x_1-y_1|^{1+sr}}\,\dd y_1\Big)^{q/r}\dd x_1,
\]
since if $y'=(y_2,\ldots y_d)$, we have
\[
    \int_{[-1,1]^{d-1}} \frac{1}{|x-y|^{d+sr}}\dd y' \gtrsim_d \frac{1}{|x_1-y_1|^{1+sr}}.
\]
Fix a nonincreasing $\phi\in C^\infty(\R)$ such that $0\le \phi\le 1$ with $\phi(t)=1$ for all $t\le -1$ 
and $\phi(t)=0$ for $t\ge 1$.
For $\varepsilon\in(0,\frac14)$ define $f_\varepsilon\colon Q\to \R$ by 
\[
f_\varepsilon(x):=\phi(x/\varepsilon).
\]
Then $f_\varepsilon(x)=1$ if $x\le -\varepsilon$ and $f_\varepsilon(x)=0$ if $x\ge \varepsilon$.
Moreover 
\[
\|f_\varepsilon'\|_{L^1(Q)}
=\int_{-1}^1 \frac{|\phi'(x/\varepsilon)|}{\varepsilon}\dd x
=\int_{-1/\varepsilon}^{1/\varepsilon}|\phi'(t)|\dd t = 1.
\]
Now fix $x\in(\varepsilon,\frac12]$, so  $f_\varepsilon(x)=0$.
For any $y\in[-1,-\varepsilon)$ we have $f_\varepsilon(y)=1$, and therefore
\[
|f_\varepsilon(x)-f_\varepsilon(y)|=1.
\]
Consequently,
\begin{align*}
\has{\int_{Q}\frac{|f_\varepsilon(x)-f_\varepsilon(y)|^r}{|x-y|^{1+sr}}\dd y}^{\frac1r}
&\ge \has{\int_{-1}^{-\varepsilon}\frac{1}{|x-y|^{1+sr}}\dd y}^{\frac1r}\\
&= \has{\int_{x+\varepsilon}^{x+1} t^{-1-sr}\dd t}^{\frac1r}\\
&\geq  \has{\int_{x+\varepsilon}^{2(x+\varepsilon)} t^{-1-sr}\dd t}^{\frac1r}\\
&=\has{\frac{1-2^{-sr}}{sr}}^{1/r} (x+\varepsilon)^{-s} \gtrsim (x+\varepsilon)^{-s}.
\end{align*}
Integrating in $x \in (\varepsilon,\frac12]$ gives
\begin{align*}
[f_\varepsilon]_{F^{s}_{q,r}(Q)}^q
&\gtrsim \int_{\varepsilon}^{\frac12}(x+\varepsilon)^{-sq}\dd x
= \frac{(\tfrac12+\varepsilon)^{1-sq}-(2\varepsilon)^{1-sq}}{1-sq}.
\end{align*}
Now take $\varepsilon =  2^{-1-\frac{2}{1-sq}}$ and set $f = f_\varepsilon$. Then we conclude 
\[
[f]_{F^{s}_{q,r}(Q)}^q\gtrsim_d \frac{\frac12 -\frac14}{1-sq} \gtrsim (1-sq)^{-1}.
\]
Taking the $q$-th root completes the proof.
\end{proof}

Precisely as in Corollary \ref{cor:AleqConeweight} and Corollary \ref{cor:AleqBoneweight}, we can specialize Theorem \ref{theorem:IIleqIII} to the one-weight case with the help of Lemma \ref{lem:ApApq}. As before, we note that the three cases in the constant below are not mutually exclusive. 
\begin{corollary} \label{cor:BleqConeweight}
    Let $Q\subseteq \R^d$ be a cube, $1\leq u$, $p_0\leq p\leq q <\infty$ and $r\in[1,\infty)$ such that $\tfrac1{p_0}-\tfrac1r\leq \tfrac1d$ and define 
    $
    \tfrac{1}{p_1}:=1-\tfrac1{p_0} + \tfrac1p.
    $
    Let $w\in A_{u}\cap A_{p/p_0}$ and $s \in (0,1)$ such that
    \begin{align*}
        \varepsilon = \tfrac{1-s}{du} -\hab{\tfrac1p - \tfrac1q} \geq 0.
    \end{align*}
   Then we have for $f\in W^{1,p}_{w^{1/p}}(Q)$
    \begin{align*}
          \has{ \frac{1}{w(Q)}\int_{Q} \has{\int_{Q} \frac{|f(x)-f(y)|^r}{|x-y|^{d+sr}}\dd y}^{\frac{q}{r}}w(x)\dd x}^{\frac{1}{q}}\lesssim_{p,q,r,d} C &\cdot [w]_{A_{p/p_0}}^{\frac{1}{p}} [w]_{A_u}^{\frac{1}{p}-\frac1q}\cdot  \ell(Q)^{1-s} \\
          &\cdot
          \has{\frac{1}{w(Q)}\int_{Q}\hab{\abs{\nabla f}}^pw \dd x}^{\frac1p},
    \end{align*}
    where
    \begin{align*}
        C:=
        \begin{cases}
            \varepsilon^{-\frac1{\min\{q,r\}}}\cdot (\varepsilon+\tfrac{s}{du})^{-1} &\varepsilon>0,\\
            \varepsilon^{-1/r}\cdot [w^{-\frac{1}{p/p_0-1}}]_{A_\infty(Q)}^{\frac1q} &\varepsilon>0, \ p>p_0,\\
            \frac{1}{s^{1+1/r}(1-s)^{1/r}}\cdot \hab{[w]_{A_\infty(Q)}^{\frac{1}{p'}}+[w^{-\frac{1}{p/p_0-1}}]_{A_\infty(Q)}^{\frac1q}} &\varepsilon=0,\ p>p_0>1 \text{ and } \tfrac1{p_0}-\tfrac1r <\tfrac1d.
        \end{cases}
    \end{align*}
\end{corollary}
\begin{proof}
    The proof is slightly different because of the $p_1$ instead of $p$ appearing in the $A^\alpha_{p_1,q}(Q)$-characteristic of $(\omega,\sigma)$ in Theorem \ref{theorem:IIleqIII}. Similar to the proof of Lemma \ref{lem:ApApq}, we use that for any cube $R \subseteq Q$  we have
	\begin{equation*}
		\has{\frac{\abs{R}}{\abs{Q}}}^u \leq [w]_{A_u}{\frac{w(R)}{w(Q)}}.
	\end{equation*}
    Therefore, using
    \[
        \frac{p_1'}{p}= \frac{p_1}{p(p_1-1)} = \frac{1}{p(1-\frac{1}{p_1})}=\frac{1}{p(\frac{1}{p_0}-\frac1p)} = \frac{1}{\frac{p}{p_0}-1}
    \]
    and taking
    \[
        \alpha=(u-1)(\tfrac1{p}-\tfrac1q),
    \]
    we have
    \begin{align*}
        \abs{R}^\alpha \ip{w^{\frac1q}}_{q,R}\ip{w^{-\frac1p}}_{p_1',R}&=\abs{R}^\alpha \ip{w}_{1,R}^{\frac1q}\ip{w^{-1}}_{\frac{1}{p/p_0-1},R}^\frac{1}{p}\\
            &\leq [w]_{A_{p/p_0}}^{1/p} \cdot \frac{\abs{R}^{\frac{u}{p}-\frac{u}q}}{w(R)^{\frac{1}{p}-\frac1q}}\leq [w]_{A_{p/p_0}}^{1/p} [w]_{A_u}^{\frac1{p}-\frac1q} \frac{|Q|^{\frac{u}{p}-\frac{u}{q}}}{w(Q)^{\frac1{p}-\frac1q}}.
		\end{align*}
        Taking $\tilde{\varepsilon}=\varepsilon u$ then gives
        \[
            [w^{1/q},w^{1/p}]_{A^{\alpha}_{p_1,q}(Q)}|Q|^{\tilde{\varepsilon}} \leq \ell(Q)^{1-s}[w]_{A_{p/p_0}}^{\frac1p} [w]_{A_u}^{\frac1{p}-\frac1q} w(Q)^{\frac1{q}-\frac1p}.
        \]
        Combining this with Theorem \ref{theorem:IIleqIII} with $\varepsilon u$  yields the result.
\end{proof}
\subsection{Comparison to the literature}\label{subs:BCliterature}
In this subsection, we compare Theorem \ref{theorem:IIleqIII} and Corollary \ref{cor:BleqConeweight} with earlier results in the literature. Our two-weight fractional Sobolev to Triebel--Lizorkin embedding   with $p\neq q$ or  $\omega^q \neq \sigma^p$ with the sharp BBM-factor $(1-s)^{\frac1r}$ seem to be entirely new. Therefore, we focus on the one-weight case, i.e. Corollary \ref{cor:BleqConeweight}, with $p=q$.
\begin{enumerate}[(i)]
\item The case $p_0=p=q=r$ and $u=1$ of Corollary \ref{cor:BleqConeweight}, and hence $\varepsilon = \frac{1-s}{d}>0$, was obtained in \cite[Corollary 6.3]{MPW24} and \cite[Corollary 2.2]{HMPV25} through entirely different methods. Indeed, in both mentioned papers the result is deduced from a more general statement where $\omega^q$ corresponds to a measure $\mu$ and $\sigma^p$ corresponds to the maximal operator applied to $\mu$, see \cite[Theorem 6.2]{MPW24} and \cite[Theorem 2.1]{HMPV25}. These more general statements fall beyond the scope of Theorem \ref{theorem:IIleqIII}, as the measure $\mu$ does not necessarily satisfy an $A_\infty$-condition. See also comparison item \ref{it:Mwcomp} in Subsection \ref{sec:ACliterature}.
\item Corollary \ref{cor:BleqConeweight} with $p=q$ is qualitatively comparable with the $k=1$ case of (the corrected version on arXiv of) \cite[Theorem 1.4]{HLYY25} (see also  its unweighted precursor in \cite{Moh24}). Indeed, in \cite[Theorem 1.4]{HLYY25} weights  $w\in A_{p}$ are allowed such that
$\tfrac{p_w}{p}-\tfrac1r < \tfrac1d$, where
$$
p_w:= \inf \cbraceb{t\in [1,\infty):w\in A_t}.
$$
Given such a weight and assuming for the moment that $p>1$, there exists a $p_w\leq u< p$ such that $\frac{u}{p}-\frac{1}{r} \leq \frac{1}{d}$ and $w \in A_u$. Then taking $p_0 = \frac{p}{u}>1$, we see that $w$ is admissible in Corollary \ref{cor:BleqConeweight}. A similar, but simpler argument 
gives comparability of the weight classes when $p=1$.
Quantitatively, Corollary \ref{cor:BleqConeweight} has the same dependence on $p,r$ and $s$ as \cite[Theorem 1.4]{HLYY25}. In addition, we have explicit dependence on the weight characteristic of $w$. 
\end{enumerate}

\section{Extensions}\label{sec:extensions}
In this final section, we briefly comment on several directions in which the two-weight fractional Poincar\'e--Sobolev sandwich in this article may be extended. 
\begin{enumerate}[(i)]
    \item \emph{Higher-order derivatives.}
A natural question is whether our results extend to derivatives of order $m\geq 2$. In that setting, one would replace the first-order oscillation $f-\ip{f}_{Q}$
by the remainder obtained after subtracting a polynomial of degree $m-1$, and
then aim to establish a higher-order two-weight Poincar\'e--Sobolev sandwich,
with $\abs{\nabla^m f}$ on the right-hand side and an intermediate fractional
seminorm involving higher-order differences. Let us briefly comment on the
availability of the ingredients used in the proofs of our main theorems:
\begin{itemize}
    \item The domination principles used in the subcritical cases of our main
    theorems, namely \eqref{eq:basicdom} and Lemma \ref{lemma:Tldomsubcrit},
    extend directly to this setting.

    \item The required higher-order analogue of the sparse domination principle
    in Lemma \ref{lem:sparsedom} is available in
    \cite[Proposition 5.4]{LLO21}. We also expect that
    Theorem \ref{thm:fracsparsedom} extends to this higher-order setting.

    \item It is well known that the classical unweighted Poincar\'e inequality
    extends to higher-order derivatives. Moreover, the fractional Poincar\'e inequality in
    Lemma \ref{lem:fracbasic} is proved for higher-order differences in the
    reference we used, namely \cite{DLTYY24}.
\end{itemize}
    \item \emph{Banach function spaces beyond $L^p$.}
Another natural direction is to replace the weighted $L^p$-spaces by suitable
Banach function spaces \cite{LN24}. In this setting, the main tool would be the
boundedness of the (fractional) Hardy--Littlewood maximal operator, which fits
well with our sparse domination approach. For this reason, we expect that the
critical cases of our main theorems admit analogues in the setting of Banach
function spaces. By contrast, the proofs of the subcritical cases, as well as
the truncation method, rely heavily on the specific
structure of $L^p$-spaces. For recent work in this direction (typically
corresponding to analogues of the case $p=q$ in our setting) we refer to \cite{DGP23,CGYY25,ZYY24} and the references therein.
    \item \emph{BBM-type formulas.}
The sharp factor $(1-s)^{1/r}$ in our fractional Poincar\'e--Sobolev estimate and Sobolev to Triebel--Lizorkin embedding
is motivated by Bourgain--Brezis--Mironescu type limits as $s\uparrow 1$
\cite{BBM01,BBM02} (see also \cite{DM23} for an interpolation theory perspective). It would be interesting to understand what kind of weighted
BBM formula can be deduced from our inequalities. In particular,
Theorem \ref{theorem:IIleqIII} suggests a limit formula
\begin{align*}
   \lim_{(s,q)\to(1,p)} (1-s)^{1/r}\,[f]_{F^{s,\sigma}_{q,r}(Q)}
   = \nrmb{|\nabla f|}_{L^{p}_{\sigma}(Q)}
\end{align*}
under suitable assumptions on $\sigma,p,q,r$, and $s$. For $p=q$, such a formula can be found in \cite{DGP23}. For $q\neq p$,
such a statement appears to be new even in the unweighted setting.
   \item \emph{More general domains $\Omega \subseteq \R^d$}. We have restricted ourselves to cubes in order to isolate the main two-weight phenomena. 
However, the results are expected to hold on more general domains under the usual geometric hypotheses, e.g. on John domains.  In that setting, one decomposes \(\Omega\) into Whitney cubes, applies the cube inequality on each cube, and then patches the resulting estimates together along Whitney chains, see, for instance, \cite{DRS10}.
For the fractional Poincar\'e--Sobolev inequalities on John domains, see \cite{DIV16,HV13}.    
    \item \emph{Global inequalities on $\R^d$.} The critical cases of our main results on cubes $Q\subseteq \R^d$ should extend
to $\R^d$ by passing to the limit along an exhausting sequence of cubes. In
this case, the averages $\ip{f}_Q$ on the left-hand side converge to zero.
For the fractional seminorm, there are at least two natural global formulations.
Indeed, one may either let both cubes in the seminorm grow, or let only the
outer cube in the weighted Lebesgue norm grow while keeping the inner
oscillation localized.

It would be interesting to compare such global inequalities with the existing
literature on weighted fractional Poincar\'e--Sobolev inequalities on $\R^d$.
Moreover, unlike in the $Q$-localized setting, on $\R^d$ one can also study the
asymptotic regime $s\downarrow 0$. In the unweighted case, this regime is
governed by the Maz'ya--Shaposhnikova formula \cite{MS02}; see also
\cite{DM23,PYYZ24}.  
\end{enumerate}

\appendix

\section{The truncation method}
In this section we prove two weak implies strong principles used in the critical cases of Theorem \ref{theorem:IleqIII} and Theorem \ref{theorem:IleqII}. For both proofs, for $t>0$ and a nonnegative function $v$ we define the truncated function
    \begin{align*}
        v_t&:=\begin{cases}
        0 \qquad & v(x)\leq t,\\
                v(x)-t \qquad & t < v(x) < 2t,\\
                t\qquad & v(x)\geq 2t.\\
            \end{cases}
    \end{align*} 
    We start by stating a weak implies strong principle for the two-weight Poincar\'e--Sobolev inequality with unweighted averages on the left-hand side. The proof is a combination of \cite[Theorem 2]{Ha01} and \cite[Theorem 3.2]{FPW98}.
\begin{proposition}\label{prop:weaktostrongclassic}
        Let $Q\subseteq \R^d$ be a cube, $1\leq p\leq q$ and let $\omega,\sigma$ be weights with $\sigma^{-1}\in L^{p'}(Q)$. Suppose that there is a constant $C$ such that for all $u\in W^{1,p}_\sigma(Q)$ we have 
            \begin{align*}
                \nrm{u-\ip{u}_Q}_{L^{q,\infty}_\omega(Q)} &\leq C \,\nrm{\nabla u}_{L^p_\sigma(Q)},\\
               \frac{ \omega^q(Q)^{\frac1q}}{|Q|}\cdot   \nrm{u-\ip{u}_Q}_{L^1(Q)} &\leq  C  \,\nrm{\nabla u}_{L^p_\sigma(Q)}.
            \end{align*}
			Then for all $u\in W^{1,p}_\sigma(Q)$ we have 
			\begin{align*}
				\nrm{u-\ip{u}_Q}_{L^{q}_\omega(Q)} \leq 10C \,\nrm{\nabla u}_{L^p_\sigma(Q)}.
			\end{align*}
\end{proposition}
\begin{proof}
    Let $u\in  W^{1,p}_\sigma(Q)$ and let $\lambda>0$ be a constant to be chosen later and for $k\in\Z_{\geq -1}$ define $\lambda_k = 2^k \lambda$. To ease notation, define $v=|u-\ip{u}_Q|$. Let $E_k := \{x\in Q: v(x)> \lambda_k\}$ and $$A_k := E_{k-1}\backslash E_k = \{x\in Q: \lambda_{k-1}<v(x)\leq  \lambda_k\}.$$    Then we have
    \begin{align}
        \int_Q v^q \omega^q &= \int_{\cbrace{v \leq 2\lambda}} v^q \omega^q + \sum_{k=1}^\infty \int_{A_{k+1}} v^q\omega^q \leq 2^q\lambda^q \omega^q(Q) + \sum_{k=1}^\infty\lambda_{k+1}^q \omega^q(A_{k+1}). \label{eq:vqomegaqsplit}
    \end{align}
    Now note that if $x\in A_{k+1}$, we have 
    \begin{align*}
        \lambda_{k-1}&= v_{\lambda_{k-1}}(x) \leq  |v_{\lambda_{k-1}}(x)-\ip{v_{\lambda_{k-1}}}_Q| + \ip{v_{\lambda_{k-1}}}_Q \\
        &\leq |v_{\lambda_{k-1}}(x)-\ip{v_{\lambda_{k-1}}}_Q| + \tfrac12\ip{v}_Q,
    \end{align*}
    since $2v_t \leq v$. Let us choose $\lambda = \ip{v}_Q$. If $\lambda=0$, then $v=0$ and the statement is trivial. Therefore we can assume $\lambda >0$. For $k\geq 1$ we have
    \[
        \lambda_{k-1} \leq  |v_{\lambda_{k-1}}(x)-\ip{v_{\lambda_{k-1}}}_Q| + \tfrac12 \lambda \leq |v_{\lambda_{k-1}}(x)-\ip{v_{\lambda_{k-1}}}_Q| + \tfrac12\lambda_{k-1}
    \]
    and thus
    \[
        \lambda_{k-2} \leq |v_{\lambda_{k-1}}(x)-\ip{v_{\lambda_{k-1}}}_Q|.
    \]
    Therefore, using our two assumptions and \eqref{eq:vqomegaqsplit},
    \begin{align*}
        \int_Q v^q \omega^q &\leq 2^q \ip{v}_Q^q \omega^q(Q) + \sum_{k=1}^\infty \lambda_{k+1}^q\omega^q(\{x\in Q:  |v_{\lambda_{k-1}}(x)-\ip{v_{\lambda_{k-1}}}_Q|\geq \lambda_{k-2}\})\\
        &\leq 2^qC^q \nrm{\nabla u}_{L^p_\sigma(Q)}^q +  8^qC^q\sum_{k=1}^\infty \nrm{\nabla  v_{\lambda_{k-1}}}_{L^p_\sigma(Q)}^q.
    \end{align*}
    Since 
    \begin{align*}
        \sum_{k=1}^\infty \int_Q |\nabla v_{\lambda_{k-1}}|^p\sigma^p = \sum_{k=1}^\infty \int_{A_k} |\nabla v|^p\sigma^p = \sum_{k=1}^\infty\int_{A_k} |\nabla u|^p\sigma^p \leq \int_Q |\nabla u|^p\sigma^p,
    \end{align*}
    the result follows from the embedding $\ell^p \hookrightarrow \ell^q$. 
\end{proof}
	Next, we show the weak implies strong principle for the two-weight fractional Poincar\'e--Sobolev inequality when $r\leq p$, again with unweighted averages on the left-hand side. We combine the proof of \cite[Theorem 4.1]{DIV16}, where the case $p=r$ is treated, again with \cite[Theorem 3.2]{FPW98}.
	\begin{proposition} \label{prop:weakstrong}
		Let $Q\subseteq \R^d$ be a cube, $1\leq r \leq  p \leq q<\infty$  and let $\omega,\sigma$ be weights with $\sigma^{-1}\in L^{p'}(Q)$. Suppose that there is a constant $C>0$ such that for all  $u\in F^{s,\sigma}_{p,r}(Q)$ we have
\begin{align*}
					\nrm{u-\ip{u}_Q}_{L^{q,\infty}_\omega(Q)} &\leq C \, [u]_{F^{s,\sigma}_{p,r}(Q)},\\
                   \frac{ \omega^q(Q)^{\frac1q}}{|Q|}\cdot \nrm{u-\ip{u}_Q}_{L^1(Q)} &\leq   C\, [u]_{F^{s,\sigma}_{p,r}(Q)}.
				\end{align*}
    Then for all $u\in F^{s,\sigma}_{p,r}(Q)$ we have  
\begin{align*}
				\nrm{u-\ip{u}_Q}_{L^{q}_\omega(Q)} \leq 58C\, [u]_{F^{s,\sigma}_{p,r}(Q)}.
			\end{align*}
			\end{proposition}
	\begin{proof} Let $u\in F^{s,\sigma}_{p,r}(Q)$ and by density assume without loss of generality that $u \in L^\infty(Q)$. Define $v=|u-\ip{u}_Q|$. Estimating $[v]_{F^{s,\sigma}_{p,r}(Q)}$ is enough since by the reverse triangle inequality we have $$[v]_{F^{s,\sigma}_{p,r}(Q)} \leq [u]_{F^{s,\sigma}_{p,r}(Q)}.$$ Following the first steps of the proof of Proposition \ref{prop:weaktostrongclassic}, we obtain
    \begin{align*}
        \int_Q v^q \omega^q &\leq 2^q \ip{v}_Q^q \omega^q(Q) + \sum_{k=1}^\infty \lambda_{k+1}^q\omega^q(\{x\in Q:  |v_{\lambda_{k-1}}(x)-\ip{v_{\lambda_{k-1}}}_Q|\geq \lambda_{k-2}\}),
    \end{align*}
    where $\lambda_k=2^k \ip{v}_Q$. Again, if $\ip{v}_Q=0$, then $v=0$ and the statement is trivial. Thus we can assume $\ip{v}_Q >0$. Let us call the first term $S$ and the second term $T$. Using our second assumption, 
    \[
             S\leq  2^qC^q [u]_{F^{s,\sigma}_{p,r}(Q)}^q.
    \]
    For the second term, we can use the first assumption of the proposition on the truncated functions $v_{\lambda_{k-1}}$ to estimate
        \begin{align*}
            T &= 8^q\sum_{k =1}^\infty \lambda_{k-2}^{q}\omega^q\hab{\{x\in Q: |v_{\lambda_{k-1}}(x)-\ip{v_{\lambda_{k-1}}}_Q| \geq \lambda_{k-2}\}} \\
            &\leq  8^qC^q \sum_{k=1}^\infty \has{\int_Q \has{\int_Q \frac{(v_{\lambda_{k-1}}(x)-v_{\lambda_{k-1}}(y))^r}{|x-y|^{d+sr}} \dd y}^{p/r}\sigma^p(x)\dd x}^{q/p} \\
            &\leq  8^qC^q \has{\sum_{k=1}^\infty \int_Q \has{\int_Q \frac{(v_{\lambda_{k-1}}(x)-v_{\lambda_{k-1}}(y))^r}{|x-y|^{d+sr}} \dd y}^{p/r}\sigma^p(x)\dd x}^{q/p},
        \end{align*}
        where we used $p\leq q$ in the last step. For $k \in \Z_{\geq 1}$ define
		\begin{align*}
			g_k(x,y)&:=\frac{|v_{\lambda_{k-1}}(x) - v_{\lambda_{k-1}}(y)|^r }{|x-y|^{d+sr}}, &&x,y \in Q, \\
			g(x,y)&:=\frac{|v(x) - v(y)|^r }{|x-y|^{d+sr}}, &&x,y \in Q.
		\end{align*}
		and, using $g_k(x,y) = 0$ for $x,y \in E_k$ or $x,y \in E_{k-1}^c$,         
        decompose 
		\begin{align*}
			\int_Q &\Big(\int_Q g_k(x,y) \dd y\Big)^{p/r}\sigma(x)^p \dx	\\&= \bigg\{\int_{A_{k}}+ \int_{E_{k-1}^c}+\int_{E_{k}}\bigg\}\Big( \int_{Q} g_k(x,y) \dd y\Big)^{p/r}\sigma(x)^p \dx  \\
			&= \int_{A_k}\Big( \int_{Q} g_k(x,y) \dd y\Big)^{p/r}\sigma(x)^p \dx + \int_{E^c_{k-1}}\Big( \sum_{j\geq {k}}\int_{A_{j}} g_k(x,y) \dd y\Big)^{p/r}\sigma(x)^p \dx  \\
			&\hspace{2cm}+ \sum_{j\geq {k+1}}\int_{A_{j}}\Big( \int_{E_{k}^c} g_k(x,y) \dd y\Big)^{p/r}\sigma(x)^p \dx\\
            &=:T_1(k)+T_2(k)+T_3(k).
		\end{align*}

For the diagonal term $T_1(k)$, we use 
\begin{align}\label{eq:Lipschitz}
    \abs{v_{\lambda_{k-1}}(x) - v_{\lambda_{k-1}}(y)} \leq \abs{v(x)-v(y)}
\end{align} to obtain
		\begin{align*}
			\has{\sum_{k=1}^\infty T_1(k)}^{1/p} \leq \has{\sum_{k=1}^\infty\int_{A_k}\Big( \int_{Q} g(x,y) \dd y\Big)^{p/r}\sigma(x)^p \dx}^{1/p} \leq [v]_{F^{s,\sigma}_{p,r}(Q)}.
		\end{align*}

		For the first off-diagonal error term $T_2(k)$, fix $j\geq k$. We note that for $x\in E_{k-1}^c$ and $y\in A_j$, we have
		\begin{align}
			\abs{v_{\lambda_{k-1}}(x) - v_{\lambda_{k-1}}(y)} < 2\cdot 2^{k-j}\abs{v(x)-v(y)}. \label{eq:weakstrongest2kj}
		\end{align}
		Indeed, for $j=k$ this follows from \eqref{eq:Lipschitz}.
 If $j\geq k+1$, then $$2\cdot v(x)\leq \lambda_{j-1} < v(y)$$ and $v_{\lambda_{k-1}}(x)=0$ and $v_{\lambda_{k-1}}(y)=\lambda_{k-1}$, so 
 \[
		\abs{v_{\lambda_{k-1}}(x) - v_{\lambda_{k-1}}(y)} = \lambda_{k-1} = 2^{k-j}\cdot   \lambda_{j-1} <  2^{k-j}\cdot v(y) < 2 \cdot 2^{k-j}\abs{v(x)-v(y)}.
		\]
    Therefore,
		\begin{align*}
			\sum_{k=1}^\infty T_2(k) &\leq \sum_{k=1}^\infty \int_{E_k^c}\Big(\sum_{j\geq k} 2^r\cdot 2^{r(k-j)}\int_{A_j} g(x,y) \dd y\Big)^{p/r}\sigma(x)^p \dx \\
			&\leq  2^p\int_{Q}\Big(\sum_{k=1}^\infty\sum_{j\geq k} 2^{r(k-j)}\int_{A_j} g(x,y) \dd y\Big)^{p/r}\sigma(x)^p \dx,
		\end{align*}
		where we used that $\ell^r\hookrightarrow \ell^p$ since $r\leq p$. Changing the order of summation gives
		\begin{align*}
			\sum_{k=1}^\infty\sum_{j\geq k} 2^{r(k-j)}\int_{A_j} g(x,y) \dd y &= \sum_{j=1}^\infty 2^{-jr}\int_{A_j} g(x,y) \dd y \cdot \sum_{1\leq k\leq j}   2^{kr} \leq \frac{1}{1-2^{-r}} \int_{Q} g(x,y) \dd y.
		\end{align*}
		Thus
		\begin{align*}
			\has{\sum_{k=1}^\infty T_2(k)}^{1/p} \leq 2 \cdot  \has{ \int_{Q}\Big(\frac{1}{1-2^{-r}} \int_{Q} g(x,y) \dd y\Big)^{p/r}\sigma(x)^p \dx}^{1/p} \leq 4 \cdot [v]_{F^{s,\sigma}_{p,r}(Q)}.
		\end{align*}
		Similarly, for the second off-diagonal error term $T_3(k)$, we have
		\begin{align*}
			\has{\sum_{k=1}^\infty T_3(k)}^{1/p}  &\leq 2 \has{\sum_{k=1}^\infty \sum_{j\geq k+1}2^{p(k-j)}  \int_{A_j} \Big( \int_{Q} g(x,y) \dd y\Big)^{p/r}\sigma(x)^p \dx}^{1/p} \\
			&\leq 2\cdot [v]_{F^{s,\sigma}_{p,r}(Q)}.
		\end{align*}
        This finishes the proof.
	\end{proof}

\bibliographystyle{alpha}
\bibliography{bibliography}
\end{document}